\documentclass[12pt,a4paper,notitlepage]{article}
\usepackage{amsthm}
\usepackage{amsfonts}
\usepackage{amsmath}
\usepackage{mathtools}

\usepackage{graphicx}

\usepackage{pdfpages}

\usepackage{cite}
\usepackage{indentfirst}
\usepackage{color}
\usepackage{hyperref}
\hypersetup{
   colorlinks,
   citecolor=black,
   filecolor=black,
   linkcolor=black,
   urlcolor=black
}

\usepackage{blindtext}
\usepackage[utf8]{inputenc}

\theoremstyle{plain}
\newtheorem{thm}{Theorem}[section] 
\newtheorem{exmp}[thm]{Example}
\newtheorem{prop}[thm]{Proposition}

\newtheorem{lem}[thm]{Lemma}

\theoremstyle{definition}
\newtheorem{defn}[thm]{Definition}

\usepackage{amssymb}

\newcommand{\RN}[1]{%
 \textup{\uppercase\expandafter{\romannumeral#1}}%
}

\newcommand{\Addresses}{{
 \bigskip
 \footnotesize

\textsc{Department of Mathematical Sciences, Durham University,
   School Road, Durham, UK,  DH1 3LE}\par\nopagebreak
 \textit{E-mail address}: \texttt{j.m.wilson2@durham.ac.uk} }}
 
 \title{Laurent phenomenon algebras arising from surfaces}
\author{Jon Wilson}
\date{}

\usepackage{fancyhdr}

\fancypagestyle{plain}{% emulate the plain style
 \fancyhf{}% clear all fields
 \fancyfoot[C]{\thepage}%
}
\fancypagestyle{firstpage}{% a rule at the bottom
 \fancyhf{}% clear all fields
 \fancyfoot[L]{\footnotesize{The author's PhD studies are funded by an EPSRC studentship.}}%
}

\usepackage[titles]{tocloft}

\usepackage{cite}

\usepackage{xcolor}
\newcommand*\circled[1]{\kern-2.5em%
  \put(0,4){\color{white}\circle*{18}}\put(-0.5,4){\circle{12}}%
  \put(-3,0){\color{black}\small#1}~~}
\usepackage{enumitem}

\usepackage{float}

\begin{document}

\maketitle

\thispagestyle{firstpage}

\begin{abstract}

It was shown by Fomin, Shapiro and Thurston \cite{fomin2008cluster} that some cluster algebras arise from orientable surfaces. Subsequently, Dupont and Palesi \cite{dupont2015quasi} extended this construction to non-orientable surfaces. We link this framework to Lam and Pylyavskyy's Laurent phenomenon algebras \cite{lam2012laurent}, showing that both orientable and non-orientable unpunctured marked surfaces have an associated LP-algebra.

\end{abstract}

\pagenumbering{arabic}

\tableofcontents

\section{Introduction}

Cluster algebras were introduced by Fomin and Zelevinsky with the intention of understanding a construction of canonical bases by Lustig and Kashiwara. Subsequently it has found deep roots in diverse areas of mathematics including Poisson geometry, integrable systems, quiver representations, polytopes and the theory of surfaces. The cluster algebra itself is a commutative ring defined by a set of generators called \textit{cluster variables}. These cluster variables are grouped into overlapping finite subsets of the same cardinality. Given a cluster there is an idea of \textit{mutation} - this broadly consists of obtaining a new cluster by substituting one of the cluster variables. The \textit{cluster structure} is the combinatorics describing how the clusters are connected via the process of mutation. In the theory of cluster algebras the main focus is usually not the underlying ring, but rather the cluster structure. In practice the set of cluster variables and clusters are not known from the outset. Instead one specifies an initial cluster together with an additional piece of combinatorial data to establish the rules of mutation - in the case of cluster algebras this data is a skew-symmetrizable matrix.  The rest of the clusters are then obtained by repeated employment of mutation.

Fomin and Zelevinsky \cite{fomin2002laurent} proved the remarkable property that every cluster variable in a cluster algebra can be written as a Laurent polynomial in the initial cluster variables. In turn, they settled Gale and Robinson's conjecture on the integrability of generalised Somos sequences, as well as several other like-minded conjectures made by Elkies, Kleber and Propp.
It is the unification of cluster algebras with the caterpillar lemma that resolve these conjectures, but cluster algebras certainly do not capture the generality which the lemma provides. Aimed at extracting the full potential out of the lemma Lam and Pylyavskyy concocted their own much broader cluster structure, which, by design, produces the Laurent phenomenon. As such, they befittingly named this structure the Laurent phenomenon algebra, or LP algebra for short.

Lam and Pylyavskyy discovered in \cite{lam2012laurent} that these LP algebras encompass cluster algebras, and also appear naturally as co-ordinate rings of Lie groups. Subsequently, Gallagher and Stevens \cite{gallagherbroken} demonstrated that their broken Ptolemy algebra exhibits an LP structure. Revealing yet more connections, in this paper, we link Dupont and Palesi's quasi-cluster algebras to LP algebras. Namely, after making a minor tweak to their definition of a quasi-triangulation, see Definition \ref{newarcs} and \ref{newcompatibilitydef}, we prove the following: \newline

\noindent \textbf{Theorem \ref{main}}. \textit{Let $(S,M)$ be an unpunctured (orientable or non-orientable) marked surface. Then the LP cluster complex $\Delta_{LP}(S,M)$ is isomorphic to the quasi-arc complex $\Delta^{\otimes}(S,M)$, and the exchange graph of $\mathcal{A}_{LP}(S,M)$ is isomorphic to $E^{\otimes}(S,M)$.} 

\textit{More explicitly, let $T$ be a quasi-triangulation of $(S,M)$ and $\Sigma_{T}$ its associated LP seed. Then in the LP algebra $\mathcal{A}_{LP}(\Sigma_{T})$ generated by this seed the following correspondence holds:}
\begin{align*}
&\hspace{8mm} \mathbf{\mathcal{A}_{LP}(\Sigma_T)} & & &\mathbf{(S,M)} \hspace{20mm}&  \\ 
&\textit{Cluster variables} &\longleftrightarrow& &\textit{Lambda lengths of quasi-arcs} & \\
&\hspace{8mm}\textit{Clusters}  &\longleftrightarrow& &\textit{Quasi-triangulations} \hspace{7mm}& \\
&\hspace{4mm} \textit{LP mutation}   &\longleftrightarrow&  &\textit{Flips} \hspace{21.5mm}& \\
\end{align*}

Dupont and Palesi's quasi-cluster algebras were an effort of extending the work of Fomin, Shapiro and Thurston \cite{fomin2008cluster} by discovering a cluster structure on non-orientable surfaces. Their setup has seeds consisting of quasi-triangulations, so the current method of generating the algebra demands directly keeping track of flips on the surface. The above theorem places the structure in the realms of LP algebras, providing a purely combinatorial description of the mutation process.

The paper is organised as follows. We begin by recalling the construction of LP algebras and quasi-cluster algebras in chapters 2 and 3, respectively. Chapter 4 makes up the bulk of the paper and is devoted to linking these two structures. Firstly we make a small alteration to the definition of a quasi-cluster algebra as suggested by Pylyavskyy in private communication \cite{pylyavskyy2016}. This change is in keeping with the flavour of cluster algebras and only alters the cluster structure - the underlying ring is not affected by it. Next, by considering the (orientable) double cover of the marked surface $(S,M)$ we restrict our attention to the quasi-triangulations that lift to triangulations, and we consider their adjacency quivers. By using the \textit{anti-symmetric} property of these quivers we show that LP mutation agrees with quasi-cluster mutation when mutating amongst this type of quasi-triangulation. From here, through a case by case check, we show LP and quasi-cluster mutation agree everywhere.

\section*{\large \centering Acknowledgements}
I would like to thank Pavlo Pylyavskyy for suggesting the modified definition of a quasi-cluster algebra used in Section 4, and for many more of his valuable comments. I also wish to thank Anna Felikson and Pavel Tumarkin for countless helpful discussions whilst writing this paper.

\section{Laurent phenomenon algebras}
\label{Lp algebras}

This chapter follows the work of Lam and Pylyavskyy \cite{lam2012laurent}. \newline

Let $R$ be a unique factorisation domain over $\mathbb{Z}$ and let $\mathcal{F}$ be the rational field in $n \geq 1$ independent variables over the field of fractions $Frac(R)$. \newline

A Laurent phenomenon (LP) \textit{\textbf{seed}} in $\mathcal{F}$ is a pair $(\textbf{x}, \textbf{F})$ satisfying the following conditions:

\begin{itemize}

\item $\textbf{x} = \{x_1, \ldots, x_n\}$ is a transcendence basis for $\mathcal{F}$ over $Frac(R)$.

\item $\textbf{F} = \{F_1, \ldots, F_n\}$ is a collection of irreducible polynomials in $R[x_1, \ldots, x_n]$ such that for each $i \in \{1, \ldots, n\}$, $F_i \notin \{x_1, \ldots, x_n\}$; and $F_i$ does not depend on $x_i$ .

\end{itemize} 

Just as in cluster algebras, $\textbf{x}$ is called the \textit{\textbf{cluster}} and $x_1, \ldots, x_n$ the \textit{\textbf{cluster variables}}. $F_1, \ldots, F_n$ are called the \textit{\textbf{exchange polynomials}}. \newline

Recall that a cluster algebra seed of geometric type $(\textbf{x}, B)$ consists of a cluster $\textbf{x}$ and a skew symmetric matrix $B = (b_{ij})$. We can recode this matrix into binomials defined by $F^B_j := \prod_{b_{ij}>0}x_i^{b_{ij}} + \prod_{b_{ij}<0}x_i^{-b_{ij}}$, so there is a strong similarity between the definition of cluster algebra and LP seeds. The key difference being that for LP our exchange relations can be polynomial, not just binomial. However, unlike in cluster algebras, these polynomials are required to be irreducible. \newline
\indent To obtain an \textit\textbf{LP algebra} from our seed we imitate the construction of cluster algebras. Namely, we introduce a notion of mutation of seeds. Our LP algebra will then be defined as the ring generated by all the cluster variables we obtain throughout the mutation process. Before we present the rules of mutation we first need to introduce the idea of normalising exchange polynomials and clarify notation.

\underline{\textbf{Notation:}} \begin{itemize}

\item Let $F, G$ be Laurent polynomials in the variables $x_1, \ldots x_n$. We denote by $F \rvert_{x_j \leftarrow G}$ the expression obtained by substituting $x_j$ in $F$ by the Laurent polynomial $G$.

\item If $F$ is a Laurent polynomial involving a variable $x$ then we write $x \in F$. Likewise, $x \notin F$ indicates that $F$ does not involve $x$.

\end{itemize}

\begin{defn}

Given $\mathbf{F} = \{F_1,\ldots, F_n\}$ then for each $j \in \{1,\ldots,n\}$ we define $\hat{F}_j := \frac{F_j}{x_1^{a_1}\ldots x_{j-1}^{a^{j-1}}x_{j+1}^{a_{j+1}}\ldots x_n^{a_n}}$ where $a_k \in \mathbb{Z}_{\geq 0}$ is maximal such that $F_k^{a_k}$ divides ${F}_j \rvert_{x_k \leftarrow \frac{F_k}{x}}$, as an element of $R[x_1, \ldots, x_{k-1},x^{-1},x_{k+1},\ldots,x_n]$.  The Laurent polynomials of $\mathbf{\hat{F}} := \{\hat{F}_1,\ldots, \hat{F}_n\}$ are called the \textbf{\textit{normalised exchange polynomials}}.
\end{defn}

\begin{exmp}
\label{norm}
Consider the following exchange polynomials in $\mathbb{Z}[a,b,c]$ $$F_a = b+1, F_b = a+c, F_c = (b+1)^2 + a^2b.$$ Since $a$ appears in both $F_b$ and $F_c$ then $\hat{F}_a = F_a$ (see Lemma \ref{welldefined}). Similarly, $\hat{F}_b = F_b$. As $c \in F_b$ then $b \notin \frac{F_c}{\hat{F}_c}$. However, $2$ is the maximal power of $F_a$ that divides $F_{c} \rvert_{a \leftarrow \frac{F_a}{x}}$, so $\hat{F}_c = \frac{F_c}{a^2}$.

\end{exmp}

\begin{defn}

Let $(\mathbf{x}, \mathbf{F})$ be a seed and $i \in \{1,\ldots, n\}$. We define a new seed $\mu_i(\mathbf{x}, \mathbf{F}) := (\{x_1',\ldots,x_n'\},\{F_1',\ldots, F_n'\})$. Here $x_j' = x_j$ for $j \neq i$ and $x_i' = \hat{F}_i/x_i$.  The exchange polynomials change as follows:

\begin{itemize}

\item If $x_i \notin F_j$ then $F_j' := F_j$.

\item If $x_i \in F_j$ then $F_j'$ is obtained by following the 3 step process outlined below.

\begin{description}[align=left]
\item [(Step 1)] Define $G_j := F_j \rvert_{x_i \leftarrow \frac{\hat{F}_i \rvert_{x_j \leftarrow 0}}{x_i'}}$ 
\item [(Step 2)] Define $H_j := (G_j$ with all common factors with $\hat{F}_i \rvert_{x_j \leftarrow 0}$ divided out). I.e. we have $gcd(H_j, \hat{F}_i \rvert_{x_j \leftarrow 0}) =1$.
\item [(Step 3)] Let $M$ be the unique monic Laurent monomial in $R[x_1'^{\pm 1},\ldots, x_n'^{\pm 1}]$ such that $F_j' := H_jM \in R[x_1',\ldots, x_n']$ and is not divisible by any of the variables $x_1', \ldots, x_n'$.
\end{description}

\end{itemize}

The new seed $\mu_i(\mathbf{x}, \mathbf{F})$ is called the \textit{\textbf{mutation}} of $(\mathbf{x}, \mathbf{F})$ in \textit{\textbf{direction $\boldsymbol{i}$}}. It is important to note that because of \textbf{Step 2} the new exchange polynomials are only defined up to a unit in $R$.

\end{defn}

It is certainly not clear a priori that $\mu_k(\mathbf{x}, \mathbf{F})$ will be a valid LP seed due to the irreducibility requirement of the new exchange polynomials. Furthermore, due to the expression $\hat{F}_k \rvert_{x_i \leftarrow 0}$ appearing in \textbf{Step 1} it may not be apparent that the process is even well defined. These issues are resolved by the following two lemmas.

\begin{lem}[Proposition 2.7, \cite{lam2012laurent}]
\label{welldefined}
$x_k \in F_i \implies  x_i \notin \frac{F_k}{\hat{F}_k}$. In particular, $x_k \in F_i$ implies that $\hat{F}_k \rvert_{x_i \leftarrow 0}$ is well defined.

\end{lem}

\begin{lem}[Proposition 2.15, \cite{lam2012laurent}]

$F_i'$ is irreducible in $R[x_1',\ldots, x_n']$ for all $i \in \{1,\ldots,n\}$. In particular, $\mu_k(\mathbf{x}, \mathbf{F})$ is a valid LP seed.

\end{lem}

\begin{exmp}

We will perform mutation $\mu_{a}$ at $a$ on the LP seed $$(\{a,b,c\},\{F_a = b+1, F_b = a+c, F_c = (b+1)^2 + a^2b\}).$$ Recall from Example \ref{norm} that $\hat{F}_a = F_a$. Both $F_b$ and $F_c$ depend on $a$ so we are required to apply the 3 step process on each of them. We shall denote the new variable $a' := \frac{\hat{F}_a}{a}$ by $d$. \newline
$$G_b = F_b \rvert_{a \leftarrow \frac{\hat{F_a} \rvert_{b \leftarrow 0}}{d}} = F_b \rvert_{a \leftarrow \frac{1}{d}} = \frac{c}{d} + 1.$$ Nothing happens at Step 2 since $\hat{F_a} \rvert_{b \leftarrow 0} = 1$. Multiplying by the monomial $d$ gives us our new exchange polynomial $F_b' = c + d$. \newline
$$G_c = F_c \rvert_{a \leftarrow \frac{\hat{F_a} \rvert_{c \leftarrow 0}}{d}} = F_c \rvert_{a \leftarrow \frac{b+1}{d}} = (b+1)^2 + \frac{(b+1)^2b}{d^2}.$$ Following Step 2 we divide $G_c$ by any of its common factors with $\hat{F_a} \rvert_{c \leftarrow 0} = b+1$. This leaves us with $H_c = 1 + \frac{b}{d^2}$. Finally, multiplying by the monomial $d^2$ gives us our new exchange polynomial $F_c' = d^2 + b$. \newline
Hence, our new LP seed is $$(\{d,b,c\},\{F_d = b+1, F_b = c+d, F_c = d^2 + b\}).$$

\end{exmp}

Recall that mutation in cluster algebras is an involution. In the LP setting, because mutation of exchange polynomials is only defined up to a unit in $R$, it is clear we can't say precisely the same thing for LP mutation. Nevertheless, we do have the following analogue.

\begin{prop}[Proposition 2.16, \cite{lam2012laurent}]

If $(\mathbf{x}',\mathbf{F}')$ is obtained from $(\mathbf{x},\mathbf{F})$ by mutation at $i$, then $(\mathbf{x},\mathbf{F})$ can be obtained from $(\mathbf{x}',\mathbf{F}')$ by mutation at $i$. It is in this sense that LP mutation is an involution.
\end{prop}

\begin{defn}

A \textbf{\textit{Laurent phenomenon algebra}} $(\mathcal{A}, \mathcal{S})$ consists of a collection of seeds $\mathcal{S}$, and a subring $\mathcal{A} \subset \mathcal{F}$ that is generated by all the cluster variables appearing in the seeds of $\mathcal{S}$. The collection of seeds must be connected and closed under mutation. More formally, $\mathcal{S}$ is required to satisfy the following conditions: \begin{itemize}

\item Any two seeds in $\mathcal{S}$ are connected by a sequence of LP mutations.
\item$\forall$ $(\mathbf{x},\mathbf{F}) \in \mathcal{S}$ $\forall i \in \{1,\ldots,n\}$ there is a seed $(\mathbf{x}',\mathbf{F}') \in \mathcal{S}$ that can be obtained by mutating $(\mathbf{x},\mathbf{F})$ at $i$.

\end{itemize}

\end{defn}

\begin{defn}[Subsection 3.6, \cite{lam2012laurent}]

The \textit{\textbf{cluster complex}} $\Delta_{LP}(\mathcal{A})$ of an LP algebra $\mathcal{A}$ is the simplicial complex with the ground set being the cluster variables of $\mathcal{A}$, and the maximal simplices being the clusters.

\end{defn}

\begin{defn}[Subsection 3.6, \cite{lam2012laurent}]

The \textit{\textbf{exchange graph}} of an LP algebra $\mathcal{A}$ is the graph whose vertices correspond to the clusters of $\mathcal{A}$. Two vertices are connected by an edge if their corresponding clusters differ by a single mutation.

\end{defn}

\section{Quasi-cluster algebras}
\label{quasi}

This chapter follows the work of Dupont and Palesi \cite{dupont2015quasi}. \newline

Let $S$ be a compact $2$-dimensional manifold with boundary $\partial S \neq \varnothing$. Fix a finite set $M$ of marked points in $\partial S$ such that each boundary component contains at least one marked point. The tuple $(S,M)$ is called a \textit{\textbf{bordered surface}}. We wish to exclude cases where $(S,M)$ does not admit a triangulation. As such, we do not allow $(S,M)$ to be a monogon, digon or a triangle. \newline

\noindent \textbf{Remark}: Note that the definition is almost identical to that given in \cite{fomin2008cluster}. The differences here are that punctures are forbidden, and now $S$ can be non-orientable. We omit punctured surfaces from the outset because their associated (quasi)-cluster structure has no LP structure. We conclude Chapter 4 by discussing why this is the case. \newline

To imitate the construction of cluster algebras arising from orientable surfaces we must first agree on which curves will form our notion of 'triangulation'.

\begin{defn}

An \textit{\textbf{arc}} of $(S,M)$ is a simple curve in $S$ connecting two (not necessarily distinct) marked points of $M$. 

\end{defn}

\begin{defn}

A simple closed curve in $S$ is said to be \textit{\textbf{two-sided}} if it emits a regular neighbourhood which is orientable. Otherwise, it is said to be \textit{\textbf{one-sided}}.

\end{defn}

\begin{defn}

A \textit{\textbf{quasi-arc}} is either an arc or a one-sided closed curve. Throughout this paper we shall always consider quasi-arcs up to isotopy. Let $A^\otimes(S,M)$ denote the set of all quasi-arcs (considered up to isotopy).

\end{defn}

Recall that a closed non-orientable surface is homeomorphic to the connected sum of $k$ projective planes $\mathbb{R}P^2$. Such a surface is said to have (non-orientable) genus $k$. A \textit{\textbf{cross-cap}} is a cylinder where antipodal points on one of the boundary components are identified. In particular, note that a cross-cap is homeomorphic to $\mathbb{R}P^2$ with an open disk removed. An illustration of a cross cap in given in Figure \ref{crosscap} - throughout this paper we shall always represent it in this way.
For pictorial convenience we use the following alternative description: A compact non-orientable surface of genus $k$ (with boundary) is homeomorphic to a sphere where more than $k$ open disks are removed, and $k$ of them have been replaced with cross-caps.

\begin{figure}[H]
\centering
\includegraphics[width=30mm]{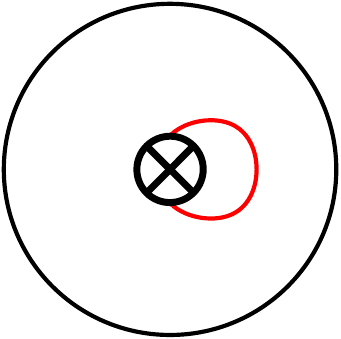}
\caption{A picture of a crosscap together with a one-sided closed curve.}
\label{crosscap}
\end{figure}

\begin{defn}

Two quasi-arcs of $(S,M)$ are called \textit{\textbf{compatible}} if there exists representatives in their respective isotopy classes that do not intersect in the interior of $S$.

\end{defn}

\begin{defn}

A \textit{\textbf{quasi-triangulation}} of $(S,M)$ is a maximal collection of pairwise compatible quasi-arcs of $(S,M)$.

\end{defn}

\noindent \textbf{Remark}: After putting a hyperbolic metric on $(S,M)$ we need only ever consider the geodesic representatives of quasi-arcs due to the fact that quasi-arcs are compatible \textit{if and only if} their geodesic representatives do not intersect.

\begin{prop}[Proposition 2.4, \cite{dupont2015quasi}]
\label{flip}
Let $T$ be a quasi-triangulation of $(S,M)$. Then for any $\gamma \in T$ there exists a unique $  \gamma' \in A^\otimes(S,M)$ such that $\gamma' \neq \gamma$ and $\mu_{\gamma}(T) := T\setminus\{\gamma\}\cup \gamma'$ is a quasi-triangulation.

\end{prop}

Here $\mu_{\gamma}(T)$ is called the \textit{\textbf{quasi-mutation}} of $T$ in \textit{\textbf{direction \boldmath$\gamma$}}, and $\gamma'$ is called the \textit{\textbf{flip}} of $\gamma$ with \textit{\textbf{respect to $\mathbf{T}$}}. The \textit{\textbf{flip graph}} of a bordered surface $(S,M)$ is the graph with vertices corresponding to quasi-triangulations and edges corresponding to flips.

\begin{prop}[Prop. 2.12, \cite{dupont2015quasi}]
\label{flipconnected}
The flip graph of $(S,M)$ is connected.

\end{prop}

Note that Propositions \ref{flip} and \ref{flipconnected} tell us that the number of quasi-arcs in a quasi-triangulation is an invariant of $(S,M)$ - this number is called the \textit{\textbf{rank}} of $(S,M)$.
\newline

\indent We now introduce the notion of a seed of a bordered surface $(S,M)$.

\subsection*{Quasi-seeds and mutation.}

Suppose $(S,M)$ is a bordered surface of rank $n$ and let $b_1,\ldots, b_m$ consist of all the boundary segments of $(S,M)$. Denote $\mathcal{F}$ as the rational field of functions in $n+m$ independent variables over $\mathbb{Q}$.

A \textit{\textbf{quasi-seed}} of a bordered surface $(S,M)$ in $\mathcal{F}$ is a pair $(\mathbf{x},T)$ such that:

\begin{itemize}

\item $T$ is a quasi-triangulation of $(S,M)$.

\item $\mathbf{x} := \{x_{\gamma} | \gamma \in T\}$ is algebraically independent in $\mathcal{F}$ over $\mathbb{ZP} := \mathbb{Z}[x_{b_1},\ldots,x_{b_m}]$.

\end{itemize}

We call $\mathbf{x}$ the \textit{\textbf{quasi-cluster}} of $(\mathbf{x},T)$ and the variables themselves are called \textit{\textbf{quasi-cluster variables}}. \newline

To define a \textit{(quasi)-cluster structure} on $(S,M)$ we shall consider the \textit{decorated Teichm\"uller space}, $\tilde{\mathcal{T}}(S,M)$, as introduced by Penner \cite{penner2012decorated}. An element of $\tilde{\mathcal{T}}(S,M)$ consists of a complete finite-area hyperbolic structure of constant curvature $-1$ on $S\setminus M$ together with a collection of horocycles, one around each marked point. Fixing a decorated hyperbolic structure $\sigma \in \tilde{\mathcal{T}}(S,M)$, Dupont and Palesi \cite{dupont2015quasi} defined the notion of the \textit{hyperbolic length}, $\lambda_{\sigma}(\gamma)$, of a quasi-arc $\gamma$ in $(S,M)$. More explicitly, this measures the length of $\gamma$ between the horocycles at its endpoints. The \textit{lambda length}, $\lambda({\gamma})$, of a quasi-arc $\gamma$ is the evaluation map on $\tilde{\mathcal{T}}(S,M)$ sending decorated hyperbolic structures $\sigma$ to $\lambda_{\sigma}(\gamma)$. They proved the following result about lambda lengths in a fixed quasi-triangulation.

\begin{thm}[Theorem 4.2, \cite{dupont2015quasi}]

For any quasi-triangulation $T$ with quasi and boundary arcs $\gamma_1, \ldots, \gamma_{n+b}$ there exists a homeomorphism 

 \begin{align*} 
\Lambda_T \colon   \tilde{\mathcal{T}}(S & ,M) \longrightarrow \mathbb{R}_{>0}^{n+c} \\
          &\sigma \mapsto (\lambda_{\sigma}(\gamma_1), \ldots, \lambda_{\sigma}(\gamma_{n+b}))
\end{align*} 

\end{thm}

As a consequence they show that the lambda lengths of quasi-arcs and boundary arcs in a triangulation can be viewed as algebraically independent variables and we have a canonical isomorphism $$\mathbb{Q}(\{ \lambda(\gamma) | \gamma \in T\cup B(S,M) \}) \cong \mathcal{F}.$$

They define a (quasi)-cluster structure by calculating how these lambda lengths are related under flips. We provide these precise relations below in Definition \ref{lambdarules}. Note that instead of working with lambda lengths we shall instead always consider their corresponding elements in $\mathcal{F}$.

\begin{defn}
\label{lambdarules}
Given $\gamma \in T$ we define \textit{\textbf{quasi-mutation}} of $(\mathbf{x},T)$ in \textit{\textbf{direction \boldmath$\gamma$}}  to be the pair $\mu_{\gamma}(\mathbf{x},T) := (\mathbf{x}',T')$ where $T' := \mu_{\gamma}(T)$ and $\mathbf{x}' := \mathbf{x} \setminus \{ x_{\gamma} \} \cup \{x_{\gamma'}\}$. The new variable $x_{\gamma'}$ depends on the combinatorial type of flip being performed. We list below the possible flips and their corresponding variable exchange relations, which were computed in \cite{dupont2015quasi}. \newline

\noindent (1). $\gamma$ is an arc separating two different triangles. 

\begin{figure}[H]
\begin{center}
\includegraphics[width=9.5cm]{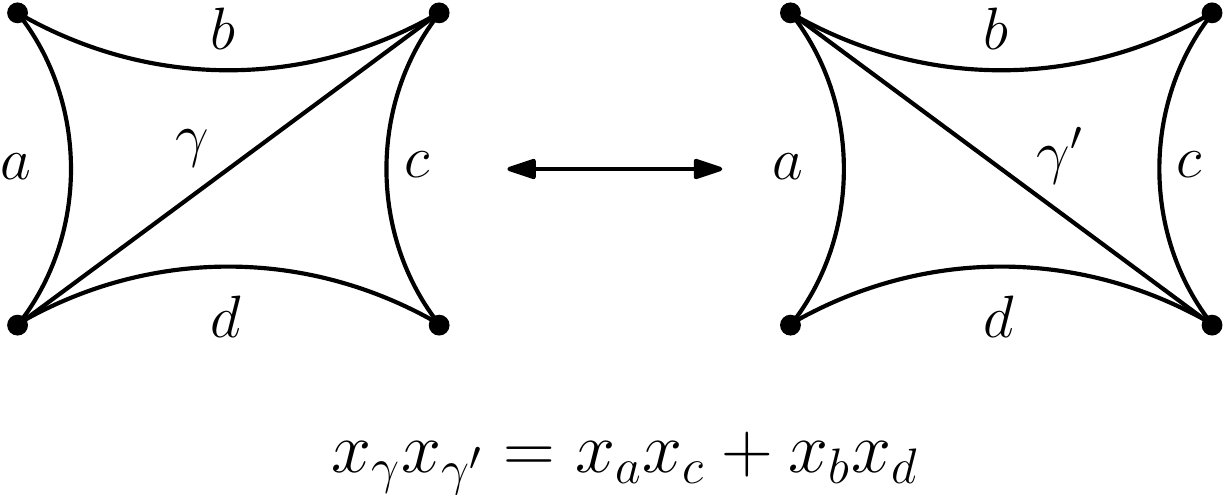}
\end{center}
\end{figure}

\noindent (2). $\gamma$ is enclosed by an arc $a$ bounding a M\"obius strip with one marked point.

\begin{figure}[H]
\begin{center}
\includegraphics[width=8cm]{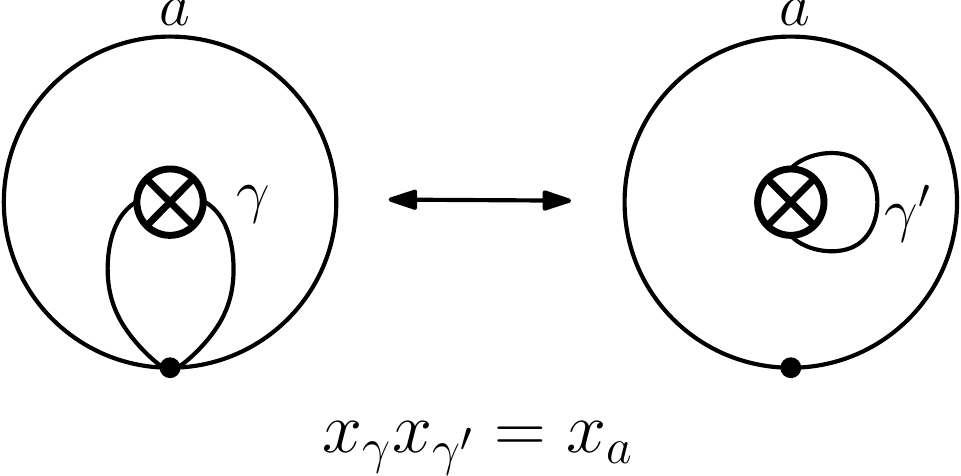}
\end{center}
\end{figure}

\noindent (3). $\gamma$ encloses a one-sided closed curve $c$.

\begin{figure}[H]
\begin{center}
\includegraphics[width=8cm]{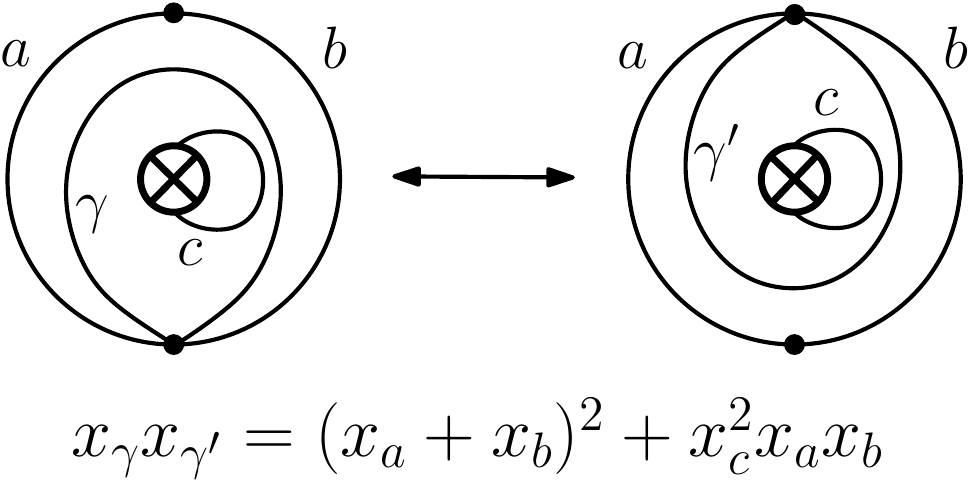}
\end{center}
\end{figure}

\end{defn}

Let $(\mathbf{x},T)$ be a seed of $(S,M)$. If we label the cluster variables of $\mathbf{x}$ $1,\ldots, n$ then we can consider the labelled n-regular tree $\mathbb{T}_n$ generated by this seed through mutations. Each vertex in $\mathbb{T}_n$ has $n$ incident vertices labelled $1,\ldots,n$. Vertices represent seeds and the edges correspond to mutation. In particular, the label of the edge indicates which direction the seed is being mutated in. \newline

Let $\mathcal{X}$ be the set of all cluster variables appearing in the seeds of $\mathbb{T}_n$. $\mathcal{A}_{(\mathbf{x},T)}(S,M) := \mathbb{ZP}[\mathcal{X}]$ is the \textit{\textbf{quasi-cluster algebra}} of the seed $(\mathbf{x},T)$.

The definition of a quasi-cluster algebra depends on the choice of the initial seed. However, if we choose a different initial seed the resulting quasi-cluster algebra will be isomorphic to $\mathcal{A}_{(\mathbf{x},T)}(S,M)$. As such, it makes sense to talk about the quasi-cluster algebra of $(S,M)$.

\section{Connecting LP algebras and quasi-cluster algebras}

\subsection{Adjusting the definition of quasi-cluster algebras.}

Recall that an LP seed must consist of irreducible polynomials. As a consequence it can be seen that, in their current form, quasi-cluster algebras can not be realised as LP algebras. (See figure below).

\begin{figure}[H]
\begin{center}
\includegraphics[width=8cm]{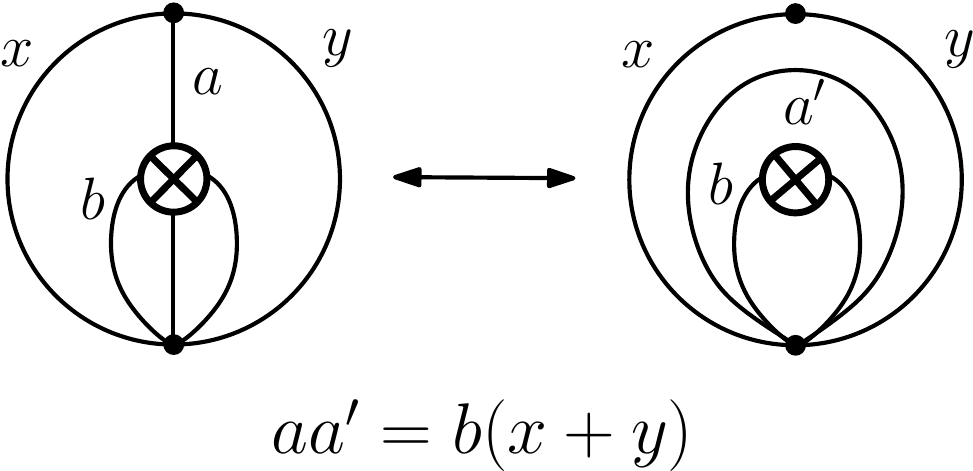}
\caption{An example of a quasi-triangulation with a reducible exchange polynomial.}
\label{badflip}
\end{center}
\end{figure}

Figure \ref{badflip} shows that in the triangulation on the left the exchange polynomial associated to the arc $a$ is $F_a = b(x+y)$, which is not irreducible in $\mathbb{Z}[x,y,a,b]$. To establish a connection between quasi-cluster algebras and LP algebras we therefore propose a small change to the quasi-arcs considered and to their compatibility relations. This alteration was suggested by Pylyavskyy in private communication \cite{pylyavskyy2016}. We shall see that the new definition is very natural - it mimics how the problem of punctured surfaces was resolved in \cite{fomin2008cluster} via tagged triangulations. \newline \indent Note that in Figure \ref{badflip} we are abusing notation by denoting the variable corresponding to an arc, by the arc itself. We shall adopt this practice from here onwards.

\begin{defn} \label{newarcs} [New definition of quasi-arcs]. A \textit{\textbf{quasi-arc}} is a one-sided closed curve or an arc that does not bound a M\"obius strip, $M_1$, with one marked point on the boundary.

\end{defn}

To each arc $\gamma$ bounding a M\"obius strip with one marked point, $M_1^{\gamma}$, we associate the two quasi-arcs of $M_1^{\gamma}$. Namely, we associate the arc $\alpha_{\gamma}$ and the one-sided curve $\beta_{\gamma}$ compatible with the $M_1^{\gamma}$, see figure below.

\begin{center}

\begin{figure}[H]
\begin{center}
\includegraphics[width=3cm]{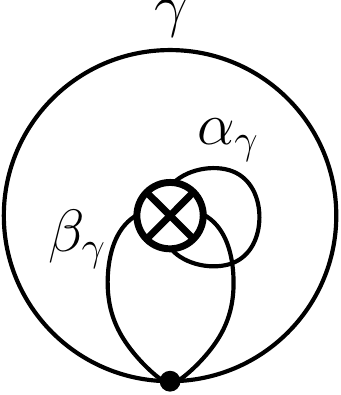}
\caption{The unique quasi-arcs $\alpha_{\gamma}$ and $\beta_{\gamma}$ compatible with the M\"obius strip, $M_1^{\gamma}$, cut out by an arc $\gamma$.}
\label{newcompatibility}
\end{center}
\end{figure}

\end{center}

\begin{defn} \label{newcompatibilitydef} [New definition of compatibility]. We say that two quasi arcs $\alpha, \beta$ are \textit{\textbf{compatible}} if they don't intersect or if $\alpha$ and $\beta$ are the two quasi-arcs of $M_1^{\gamma}$ for some arc $\gamma$. I.e, $\{\alpha,\beta\} = \{\alpha_{\gamma},\beta_{\gamma}\}$ for some arc $\gamma$ bounding a M\"obius strip $M_1^{\gamma}$ as in Figure \ref{newcompatibility}.

\end{defn}

As is usual, a \textit{\textbf{quasi-triangulation}} is a maximal collection of pairwise compatible quasi-arcs. It is easily seen that under these new definitions Proposition \ref{flip} remains true. Namely, every quasi-arc in a quasi-triangulation can be uniquely flipped.

To get a cluster structure on this new definition we imitate precisely what is done in Section 3 by describing how the lengths of quasi-arcs are related. We list below the possible types of flips and their corresponding exchange relations. Note that these relations can be directly obtained from those given in Section \ref{quasi}.  \newline

\noindent (1). $\gamma$ is an arc separating two different triangles which doesn't flip to a one-sided closed curve.

\begin{figure}[H]
\begin{center}
\includegraphics[width=9cm]{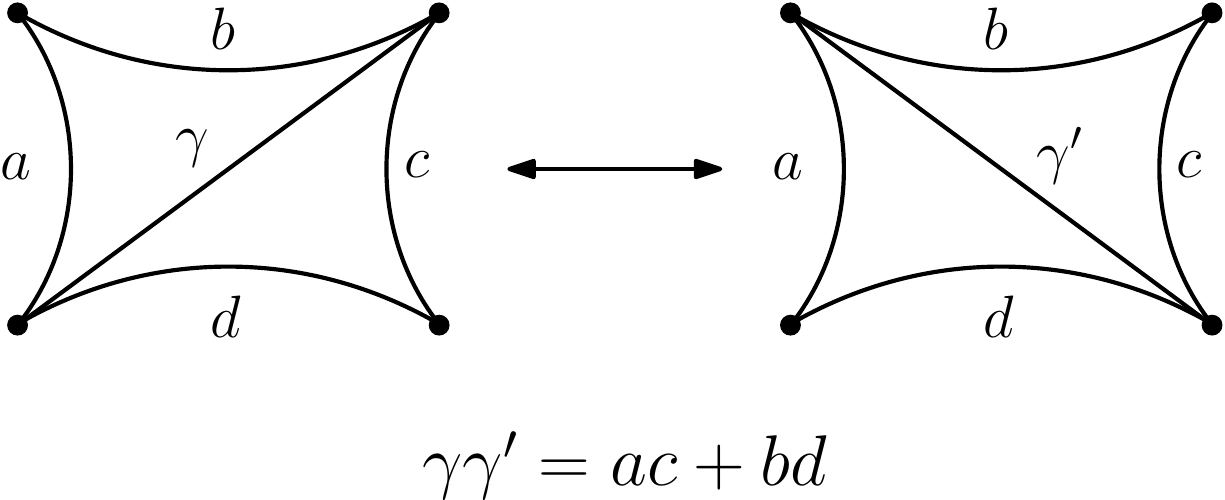}
\end{center}
\end{figure}

\noindent (2). $\gamma$ is an arc that flips to a one-sided closed curve, or vice verca.

\begin{figure}[H]
\begin{center}
\includegraphics[width=7.5cm]{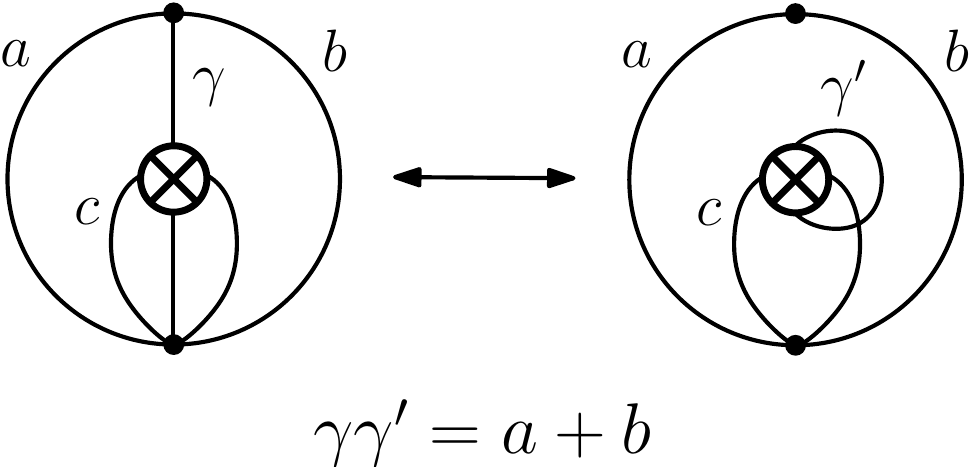}
\end{center}
\end{figure}

\noindent (3). $\gamma$ is an arc intersecting a one-sided close curve $c$.

\begin{figure}[H]
\begin{center}
\includegraphics[width=7.5cm]{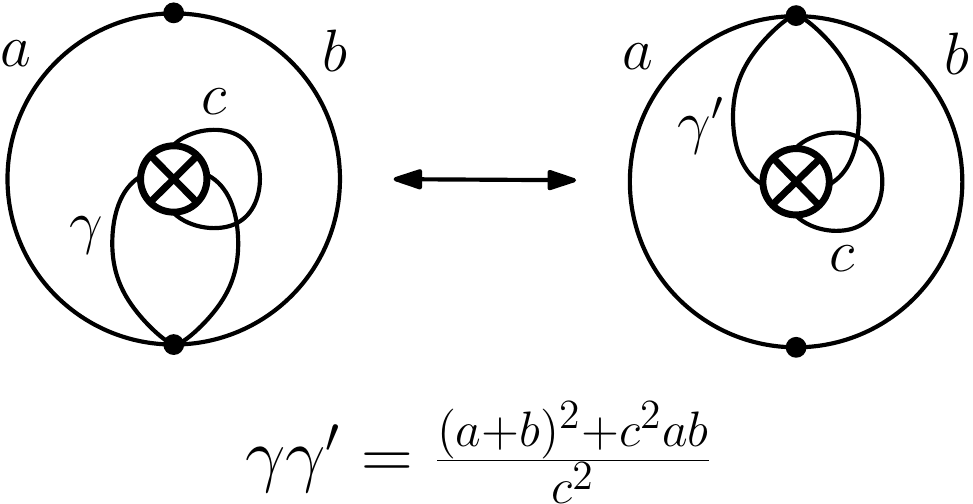}
\end{center}
\end{figure}

Recall that for our old version of quasi-cluster algebras Figure \ref{badflip} showed that for any bordered non-orientable surface $(S,M)$ (of rank greater than $1$) there exists quasi-triangulations containing reducible exchange polynomials. However, now, instead of $\gamma$ flipping to an arc bounding $M_1$, it flips to a one-side closed curve. As such, the old exchange polynomial $F_{\gamma} = b(x+y)$ has changed to the irreducible polynomial $F_{\gamma} = x+y$. \newline

\begin{defn}

The \textit{\textbf{quasi-arc complex}} $\Delta^{\otimes}(S,M)$ of a bordered surface $(S,M)$ is the simplicial complex with the ground set being the quasi-arcs of $(S,M)$, and the maximal simplices being the quasi-triangulations.

\end{defn}

\begin{defn}

The \textit{\textbf{exchange graph}} of a bordered surface $(S,M)$ is the graph whose vertices correspond to the quasi-triangulations of $(S,M)$. Two vertices are connected by an edge if their corresponding quasi-triangulations differ by a single flip.

\end{defn}

We shall now restrict our attention to quasi-triangulations not containing any one-sided closed curves. Such a quasi-triangulation will be referred to as a \textit{\textbf{triangulation}}. Furthermore, if $\gamma$ is an arc in a triangulation $T$ and $\mu_{\gamma}(T)$ is also a triangulation then we call $\gamma$ triangulation-mutable, or \textit{\textbf{t-mutable}} for short.

\subsection{The double cover and anti-symmetric quivers.}

Let $(S,M)$ be a bordered surface. We construct an orientable double cover of $(S,M)$ as follows. First consider the orientable surface $\tilde{S}$ obtained by replacing each cross-cap with a cylinder, see Figure \ref{surfaceandcylinder}.

\begin{figure}[H]
\begin{center}
\includegraphics[width=8cm]{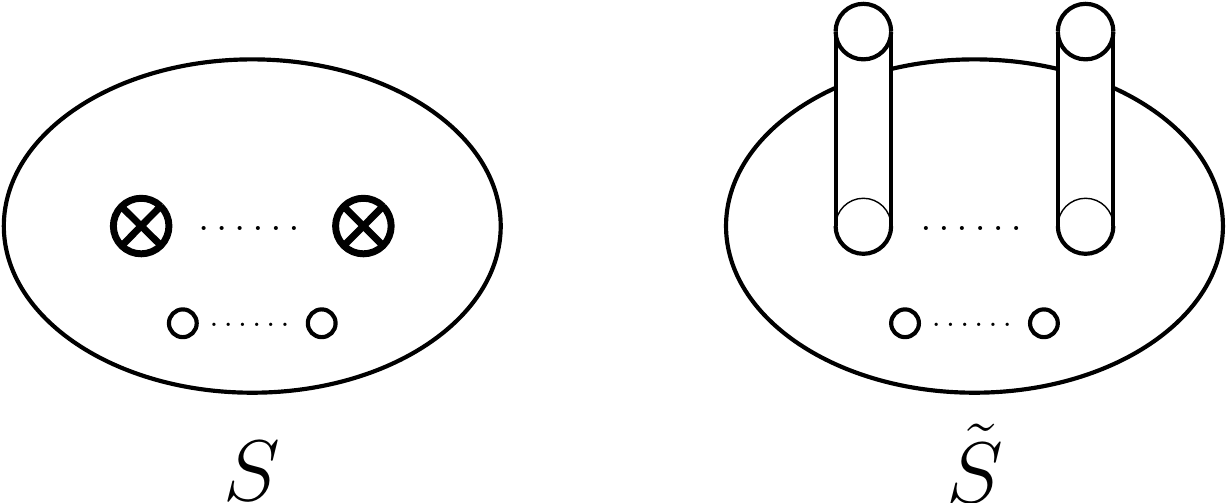}
\caption{Here we draw the non-orientable surface $S$ and the surface $\tilde{S}$ obtained by replacing each cross-cap with a cylinder. Note that the small circles represent boundary components.}
\label{surfaceandcylinder}
\end{center}
\end{figure}

We obtain the orientable double cover $\overline{(S,M)}$ of $(S,M)$ by taking two copies of $\tilde{S}$ and glueing each newly joined cylinder in the first copy, with a half twist, to the corresponding cylinder in the second copy. I.e, we are glueing each cylinder in the first copy along their antipodal points in the second copy, see Figure \ref{surfaceglueing}. If $S$ is orientable then the double cover is two disjoint copies of $(S,M)$. In this case we endow the two disjoint copies with alternate orientations - this is to ensure its adjacency quiver is anti-symmetric, see Definition \ref{antisymmetric}.

\begin{center}

\begin{figure}[H]
\begin{center}
\includegraphics[width=10cm]{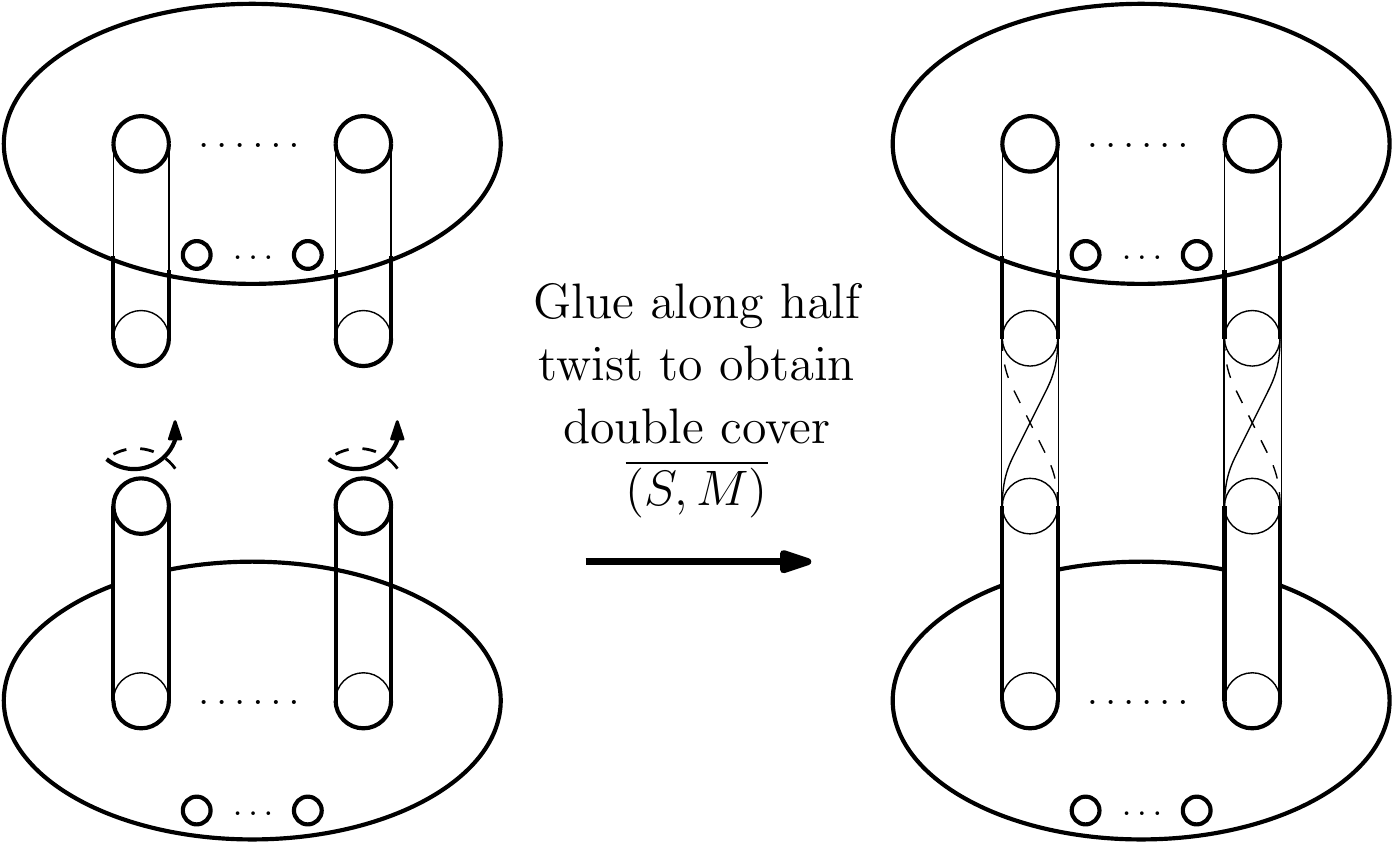}
\caption{We glue two copies of $\tilde{S}$ along the boundaries of the newly adjoined cylinders.}
\label{surfaceglueing}
\end{center}
\end{figure}

\end{center}

Due to Dupont and Palesi we have the following proposition.

\begin{prop}[\cite{dupont2015quasi}]
\label{quiverandflip}
Let $T$ be a triangulation of $(S,M)$. Then $T$ lifts to a triangulation $\overline{T}$ of the orientable double cover $\overline{(S,M)}$. Moreover, let $i$ be a t-mutable arc in $T$ and, by abuse of notation, denote by $i$ and $\tilde{i}$ the two arcs $i$ lifts to in $\overline{T}$. Then $\mu_i\circ\mu_{\tilde{i}}(\overline{T}) = \mu_{\tilde{i}}\circ\mu_i(\overline{T}) = \overline{\mu_i(T)}$.

\end{prop}

Furthermore, note that if $i$ and $j$ are arcs of a triangle $\Delta$ in $\overline{T}$, and $j$ follows $i$ in $\Delta$ under the agreed orientation of $\overline{(S,M)}$, then $\tilde{i}$ follows $\tilde{j}$ in the twin triangle $\tilde{\Delta}$. Hence in the quiver $Q_{\overline{T}}$ associated to $\overline{T}$ we have that $i \rightarrow j \iff \tilde{j} \rightarrow \tilde{i}$. Here we adopt the notation that $\tilde{\tilde{i}} = i$ for any $i \in \{1,\ldots, n\}$, and we shall use it throughout this paper.

Finally, note that there is no arrow $i \rightarrow \tilde{i}$ in $Q_{\overline{T}}$ as this would imply the existence of an anti-self folded triangle in $T$, which is forbidden under our new definition, see Figure \ref{antiself}.

\begin{figure}[H]
\begin{center}
\includegraphics[width=13.5cm]{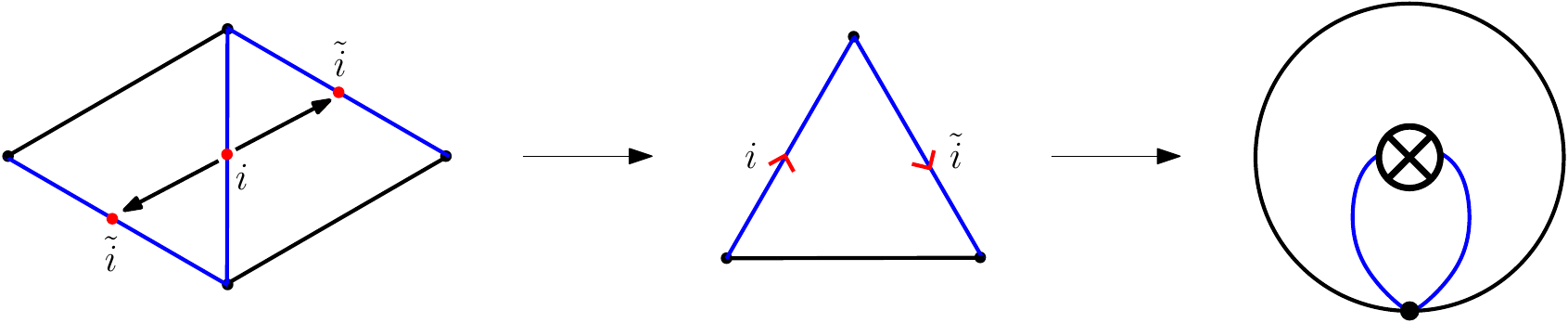}
\caption{An anti-self folded triangle; which is forbidden by the new definition.}
\label{antiself}
\end{center}
\end{figure}

These two observations motivate the following definition.

\begin{defn}
\label{antisymmetric}

A quiver $Q$ on vertices $1,\ldots, n, \tilde{1},\ldots, \tilde{n}$ is called \textit{\textbf{anti-symmetric}} if:

\begin{itemize}

\item For any $i, j \in \{1,\ldots, n, \tilde{1},\ldots, \tilde{n}\}$ we have $i \rightarrow j \iff \tilde{j} \rightarrow \tilde{i}$. 

\item  For any $i \in \{1,\ldots, n, \tilde{1},\ldots, \tilde{n}\}$ there are no arrows $i \rightarrow \tilde{i}$.

\end{itemize}

\end{defn}

\subsection{Mutation of anti-symmetric quivers as LP mutation.}

We shall now briefly leave the environment of triangulations and move to the more general setting of anti-symmetric quivers. In particular, we shall establish a connection between mutation of these quivers and LP-mutation. Recall that a quiver $Q$ can be equivalently encoded as a skew-symmetric matrix $B = (b_{ij})$. In what follows we shall interchange between the two viewpoints. 

Given an anti-symmetric quiver $Q = (b_{ij})$ we may assign an exchange polynomial to each pair of vertices $(j, \tilde{j})$ of $Q$. 

\begin{center}

 $F_j^Q := \displaystyle \prod_{b_{ij}+b_{\tilde{i}j} >0} x_i^{b_{ij}+b_{\tilde{i}j}} + \prod_{b_{ij}+b_{\tilde{i}j} <0} x_i^{-(b_{ij}+b_{\tilde{i}j})}$

\end{center}

As a result we arrive at the seed $\Sigma_Q := (\{x_1,\ldots, x_n\},\{F_1^Q,\ldots, F_n^Q\})$ associated to $Q$. Of course, this may not be a valid LP seed due to the requirement of irreducibility. We won't always get irreducibility, but, as the proposition below demonstrates, there are plenty of cases where $Q$ does provide a valid LP seed.

\begin{prop}
\label{irreducible}
If $gcd(b_{1j}+b_{\tilde{1}j},\ldots, b_{nj}+b_{\tilde{n}j}) = 1$ then $F_j$ is irreducible in $\mathbb{Z}[x_1,\ldots,x_n]$.

\end{prop}

\begin{proof}

The proof is identical to that of Lemma 4.1 in \cite{lam2012laurent}.

\end{proof}

Note that if we want double mutation of our quiver to correspond to LP mutation then it is necessary for us to have $\hat{F}_i = F_i$ $ \forall i \in \{1,\ldots, n\}$. This is because the exchange polynomials of the arcs in the triangulations are polynomials (not strictly Laurent polynomials), so the normalisation process needs to be vacuous.

\begin{prop}
\label{quivermut}
Suppose $\Sigma_Q$ is a valid LP seed and $\hat{F}_i = F_i $ $\forall i \in \{1,\ldots, n\}$. Let $i$ be a vertex in $Q$ such that there is no path $a \rightarrow i \rightarrow \tilde{a}$ for any vertex $a \in \{1,\ldots, n, \tilde{1},\ldots, \tilde{n}\}$. Then mutation at $i$ and $\tilde{i}$ in $Q$ corresponds to LP mutation of $\Sigma_Q$ at $i$. I.e, $ (\{x_1,\ldots,\frac{F_i}{x_i} ,\ldots, x_n\},\{F_1^{\mu_i\circ\mu_{\tilde{i}}(Q)},\ldots, F_n^{\mu_i\circ\mu_{\tilde{i}}(Q)}\}) = \mu_i(\{x_1,\ldots, x_n\},\{F_1^Q,\ldots, F_n^Q\})$.

\end{prop}

\begin{proof}
Let $j \in \{1,\ldots, n\}$. We will split the proof into two parts depending on whether $x_i \notin F_j^Q$ or $x_i \in F_j^Q$. \newline

If $x_i \notin F_j^Q$ then LP mutation at $i$ does not alter the exchange polynomial $F_j^Q$. I.e, $(F_j^Q)' = F_j^Q$. Therefore for quiver mutation to coincide with LP mutation we require that $F_j^{\mu_i\circ\mu_{\tilde{i}}(Q)} = F_j^Q$. It suffices to show that $$b_{kj}' + b_{\tilde{k}j}' := (\mu_i\circ\mu_{\tilde{i}}(Q))_{kj} + (\mu_i\circ\mu_{\tilde{i}}(Q))_{\tilde{k}j} = b_{kj} + b_{\tilde{k}j} \hspace{4mm} \forall k \in \{1,\ldots,n\}.$$

Below we check this holds when $k=i$ and $k \neq i$. Note that $x_i \notin F_j \implies  b_{ij}+b_{\tilde{i}j} = 0$.

\begin{itemize}

\item $(k=i)$  $b_{ij}' + b_{\tilde{i}j}' = -b_{ij} - b_{\tilde{i}j} = 0 = b_{ij} + b_{\tilde{i}j}$.

\item $(k \neq i)$ Firstly note that because mutation at $i$ and $\tilde{i}$ are independent of one another we have $$b_{kj}' := (\mu_i\circ\mu_{\tilde{i}}(Q))_{kj} = (\mu_i(Q))_{kj} + (\mu_{\tilde{i}}(Q))_{kj} - b_{kj}  =$$ $$ b_{kj} + [-b_{ki}]_{+}b_{ij} + b_{ki}[b_{ij}]_{+} + [-b_{k\tilde{i}}]_{+}b_{\tilde{i}j} + b_{k\tilde{i}}[b_{\tilde{i}j}]_{+}.$$

Now, by applying the fact that $b_{ij} = -b_{\tilde{i}j}$ we obtain the following.

$$b_{kj}'+b_{\tilde{k}j}' = b_{kj} + b_{\tilde{k}j} + b_{ij}([-b_{ki}]_+ - [-b_{k\tilde{i}}]_+ +[-b_{\tilde{k}i}]_+ -[-b_{\tilde{k}\tilde{i}}]_+) +$$ $$ [-b_{ij}]_+(b_{k\tilde{i}} + b_{\tilde{k}\tilde{i}}) + [b_{ij}]_+(b_{ki} + b_{\tilde{k}i}) \stackrel{\text{by anti-symmetry}}{=} $$ $$b_{kj} + b_{\tilde{k}j} + b_{ij}([-b_{ki}]_+ - [b_{\tilde{k}i}]_+ +[-b_{\tilde{k}i}]_+ -[b_{ki}]_+) +$$ $$ [-b_{ij}]_+(b_{k\tilde{i}} - b_{ki}) + [b_{ij}]_+(b_{ki} - b_{k\tilde{i}}).$$

Using the fact that $[a]_+ - [-a]_+ = a$ we see that $$b_{kj}'+b_{\tilde{k}j}' = b_{kj} + b_{\tilde{k}j}.$$

So indeed, $F_j^{\mu_i\circ\mu_{\tilde{i}}(Q)} = F_j^Q$ in the case $x_i \notin F_j$.

\end{itemize}

If $x_i \in F_j$ then \textit{w.l.o.g} we shall assume $b_{ij}+b_{\tilde{i}j} >0$ and $b_{ij}>0$. By skew symmetry we have $b_{ji}<0$. Also, $b_{\tilde{j}i} \leq 0$ follows from $b_{ij}>0$ and the assumption that there is no path $a \rightarrow i \rightarrow \tilde{a}$. From this we get the following:

\begin{center}

$F_i^Q |_{x_j \leftarrow 0} = \displaystyle \prod_{b_{ki}+b_{\tilde{k}i} >0} x_k^{b_{ki}+b_{\tilde{k}i}}$

\end{center}

From here we see \textbf{(Step 1)} of LP mutation gives us:

\begin{center}

$G_j^Q  = \displaystyle \Bigg( \prod_{\stackrel{b_{kj}+b_{\tilde{k}j} >0}{k \neq i}} x_k^{b_{kj}+b_{\tilde{k}j}} \Bigg) \Bigg( \frac{\prod_{b_{ki}+b_{\tilde{k}i} >0} x_k^{b_{ki}+b_{\tilde{k}i}}}{x_i'} \Bigg)^{b_{ij}+b_{\tilde{i}j}} + \prod_{b_{kj}+b_{\tilde{k}j} <0} x_k^{-(b_{kj}+b_{\tilde{k}j})}$
\end{center}

We make the observation that since $F_i^Q |_{x_j \leftarrow 0}$ is a monomial then \textbf{(Step 2)} of LP mutation can be incorporated into \textbf{(Step 3)}. Therefore to obtain $(F_j^Q)'$ we are left with the task of finding a monic Laurent monomial $M$ such that $(F_j^Q)' := MG_j^Q \in \mathbb{Z}[x_1',\ldots, x_n']$ and is not divisible by any $x_k'$. We shall determine the exponent of the variable $x_k$ in $(F_j^Q)'$ by splitting the task into three cases. For each case we check the exponent agrees with the one in the exchange polynomial $F_j^{\mu_{\tilde{i}}\circ\mu_i(Q)}$ obtained via quiver mutation. \newline

\noindent \underline{\textbf{Case 1}}: $b_{ki} + b_{\tilde{k}i} \leq 0$. \newline

This means there is no $x_k$ term in $F_i^Q |_{x_j \leftarrow 0}$. So the $x_k$ exponent remains unchanged from LP mutation. That being so, for LP mutation to agree with double quiver mutation we require that $b_{kj}' + b_{\tilde{k}j}' = b_{kj} + b_{\tilde{k}j}$. Since $b_{ki} + b_{\tilde{k}i} \leq 0$ and there is no path $k \rightarrow i \rightarrow \tilde{k}$ then $b_{ki}, b_{\tilde{k}i} \leq 0$. So $b_{jk}' = b_{kj}$, $b_{\tilde{k}j}' = b_{\tilde{k}j}$, and we therefore have agreement. \newline

\noindent \underline{\textbf{Case 2}}: $b_{ki} + b_{\tilde{k}i} > 0$ and $b_{kj} + b_{\tilde{k}j} \geq 0$. \newline

This means we get an $x_k$ term in the first monomial of $G_j^Q$, and it has exponent $b_{kj} + b_{\tilde{k}j} + (b_{ki}+b_{\tilde{k}i})(b_{ij}+b_{\tilde{i}j})$. To determine what happens with quiver mutation recall our assumption that $b_{ij} > 0$. Since there is no path $a \rightarrow i \rightarrow \tilde{a}$ for any vertex $a$ of $Q$, then $b_{ij}, b_{i\tilde{j}} \geq 0$. Likewise, because $b_{ki} + b_{\tilde{k}i} > 0 $, we get $b_{ki}, b_{\tilde{k}i} \geq 0$. Hence for quiver mutation we obtain $$b'_{kj} = b_{kj} + b_{ki}b_{ij} - b_{j\tilde{i}}b_{\tilde{i}k} $$ $$b'_{\tilde{k}j} = b_{\tilde{k}j} + b_{\tilde{k}i}b_{ij} - b_{j\tilde{i}}b_{\tilde{i}\tilde{k}}. $$

\noindent Using anti-symmetry and skew-symmetry we see $$b'_{kj} + b'_{\tilde{k}j} = b_{kj} + b_{\tilde{k}j} + (b_{ki}+b_{\tilde{k}i})(b_{ij}+b_{\tilde{i}j}) > 0$$.

\noindent Consequently, LP and quiver mutation coincide for case 2. \newline

\noindent \underline{\textbf{Case 3}}: $b_{ki} + b_{\tilde{k}i} > 0$ and $b_{kj} + b_{\tilde{k}j} \leq 0$. \newline

This means there will be an $x_k$ term in both monomials of $G_j^Q$ and after dividing out by an appropriate power of $x_k$, we are left with $x_k$ having exponent $b_{kj} + b_{\tilde{k}j} + (b_{ki}+b_{\tilde{k}i})(b_{ij}+b_{\tilde{i}j})$ in $(F_j^Q)'$. The variable $x_k$ appears in the left or right monomial of $(F_j^Q)'$ depending on whether $(b_{ki}+b_{\tilde{k}i})(b_{ij}+b_{\tilde{i}j}) \geq -(b_{kj} + b_{\tilde{k}j})$ or $(b_{ki}+b_{\tilde{k}i})(b_{ij}+b_{\tilde{i}j}) \leq -(b_{kj} + b_{\tilde{k}j})$, respectively. Just as in case 2 we observe that double mutating the quiver $Q$ yields $$b'_{kj} + b'_{\tilde{k}j} = b_{kj} + b_{\tilde{k}j} + (b_{ki}+b_{\tilde{k}i})(b_{ij}+b_{\tilde{i}j}).$$ Thus showing LP mutation agrees with double quiver mutation for case 3. \newline

Finally, in $(F_j^Q)'$ the $x_i'$ variable appears in the right monomial with exponent $b_{ij} + b_{\tilde{i}j}$. This agrees with quiver mutation since $b_{ij}' + b_{\tilde{i}j}' = -(b_{ij}+b_{\tilde{i}j}) <0$. This concludes the proof of the proposition.

\end{proof}

\subsection{Triangulations and their LP structure.}

We turn our attention back to triangulations of $(S,M)$ and show they slot into an LP structure. We achieve this by proving the adjacency quiver $Q_{\overline{T}}$ satisfies the conditions demanded in Proposition \ref{quivermut}, for each triangulation $T$ of $(S,M)$. 
Of course, we must also show that the exchange polynomials $F_1^{Q_{\overline{T}}},\ldots, F_n^{Q_{\overline{T}}}$ are the exchange polynomials of their corresponding arcs in $T$; this is settled by Lemma \ref{correctpolys}.
Note that, for triangulations of $(S,M)$ to slot into an LP structure, Proposition \ref{quivermut} requires that for each triangulation $T$ of our bordered surface we have:

\begin{itemize}

\item If $i$ is a $t$-mutable arc in $T$ then there is no path $k \rightarrow i \rightarrow \tilde{k}$ in $Q_{\overline{T}}$ for any vertex $k$.

\item The exchange polynomials $F_1^{Q_{\overline{T}}},\ldots, F_n^{Q_{\overline{T}}}$ associated to $T$ are irreducible.

\item $F_i^{Q_{\overline{T}}} = \hat{F}_i^{Q_{\overline{T}}}$ for each exchange polynomial associated to $T$.
\end{itemize}

The first two conditions are verified by Lemma \ref{breakage} and Lemma \ref{validseed}, respectively. The majority of this subsection is spent proving the third condition. We achieve this by first showing the property is equivalent to the exchange polynomials of $T$ being distinct, see Lemma \ref{nonormalisation}. From here, via Lemmas \ref{noarrows}, \ref{reverse}, \ref{cancellation} \ref{weight} \ref{equals}, we discover all bordered surfaces that emit triangulations producing non-distinct exchange polynomials. In the interest of maximal generality we allow the possibility that boundary segments do not receive variables; in which case the boundary segment is instead allocated the constant value $1$, and the corresponding vertex in the adjacency quiver is deleted.

\begin{lem}
\label{breakage}
For a triangulation $T$ of $(S,M)$ there are vertices $i,k$ of $Q_{\overline{T}}$ with $k \rightarrow i \rightarrow \tilde{k}$ \textit{if and only if} $T$ contains the M\"obius strip with two marked points, $M_2$, with $i$ being the non t-mutable arc of $M_2$. See Figure \ref{badvertex} below.

\end{lem}

\begin{proof}

To prove this lemma we reconstruct (part of) the surface $(S,M)$ using blocks. Namely, we use the quiver $Q$ to determine the adjacency of triangles in $T$. By anti-symmetry note that $k \rightarrow i \rightarrow \tilde{k}$ implies there is the path $i \leftarrow k \rightarrow \tilde{i}$. As a consequence there must be the quadrilateral $(i,a,\tilde{i},\tilde{b})$ with diagonal $k$ for some $a$ and $\tilde{b}$ not equal to $\{i,\tilde{i},k,\tilde{k}\}$, see Figure \ref{square}. By antisymmetry we also have the quadrilateral $(b,i,\tilde{a},\tilde{i})$ with diagonal $\tilde{k}$.

\begin{figure}[H]
\begin{center}
\includegraphics[width=12cm]{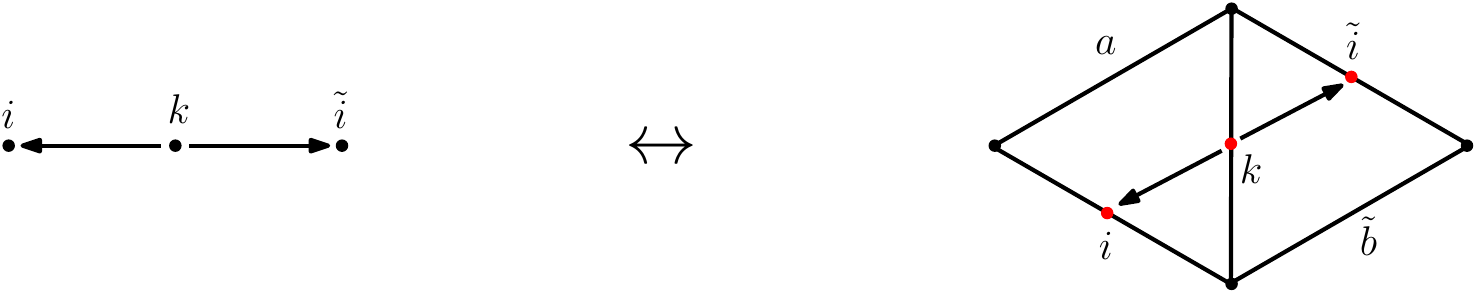}
\caption{If $i \leftarrow k \rightarrow \tilde{i}$ is a sub quiver of an adjacency quiver then the surface must have the local configuration shown on the right.}
\label{square}
\end{center}
\end{figure}

Glueing these two quadrilaterals together, according to their labels, yields the cylinder shown in Figure \ref{badvertex}. Taking the $\mathbb{Z}_2$-quotient of this leaves us with the M\"obius strip $M_2$ which is also depicted in Figure \ref{badvertex}.

\end{proof}

\begin{figure}[H]
\begin{center}
\includegraphics[width=13cm]{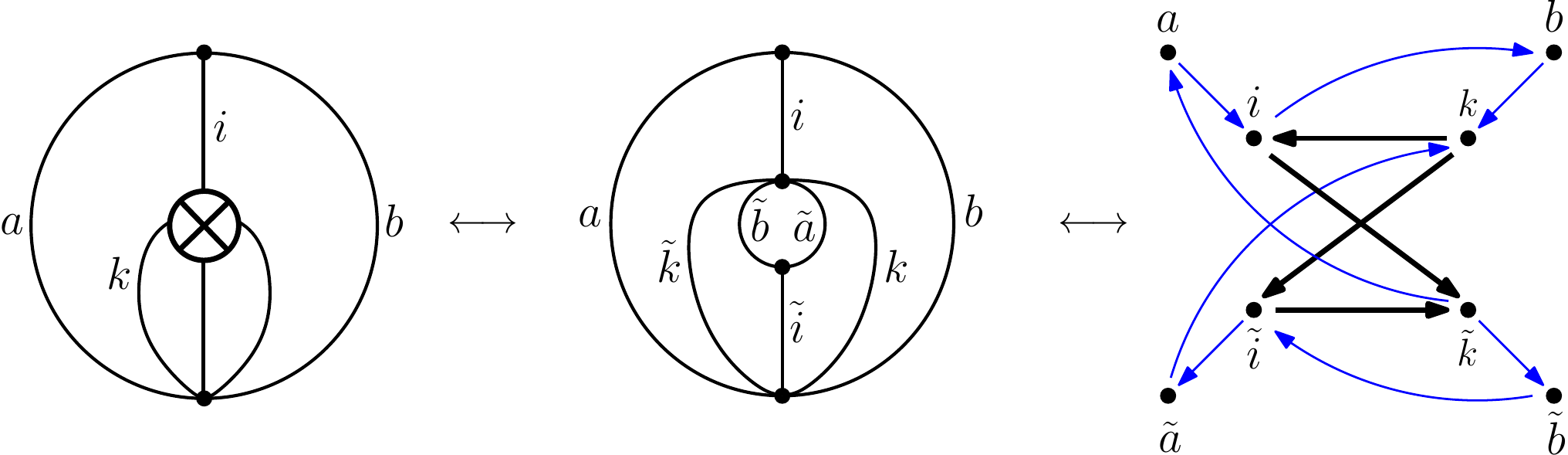}
\caption{A triangulation of the M\"obius strip $M_2$; its lifted triangulation; and the adjacency quiver of its lifted triangulation.}
\label{badvertex}
\end{center}
\end{figure}

\begin{lem}
\label{correctpolys}
Let $T$ be a triangulation of $(S,M)$ and $Q_{\overline{T}}$ the corresponding anti-symmetric quiver arising from the lifted triangulation $\overline{T}$. Then the exchange polynomials $\{F_1^{Q_{\overline{T}}},\ldots, F_n^{Q_{\overline{T}}}\}$ coincide with the exchange polynomials of the arcs in $T$ they are associated with.

\end{lem}

\begin{proof}

Let $(i,\tilde{i})$ be a twin pair of vertices in $Q_{\overline{T}}$ and consider the associated exchange polynomial $F_i^{Q_{\overline{T}}}$. If there is no path $j \rightarrow i \rightarrow \tilde{j}$ for any vertex $j$ in $Q_{\overline{T}}$ then, by Lemma \ref{breakage}, all arcs will flip to arcs. Moreover, $F_i^{Q_{\overline{T}}}:= \prod_{b_{ji}>0}x_j^{b_{ji}} + \prod_{b_{ji}<0}x_j^{-b_{ji}}$ (with the identification $x_i = x_{\tilde{i}}$) so from the standard theory of cluster algebras from surfaces we see $F_i^{Q_{\overline{T}}}$ describes how the length of the arc $i$ changes under a flip.
If there is a path $j \leftarrow i \leftarrow \tilde{j}$ then, by Lemma \ref{breakage}, locally the arc $i$ will be contained in the triangulation of $M_2$ shown in Figure \ref{badvertex}. In particular, it has the exchange polynomial $F_i^{Q_{\overline{T}}} = x_a + x_b$ which does indeed describe how the length of the arc $i$ changes under a flip. 
\end{proof}

For a seed coming from an anti-symmetric quiver $Q$ we noted that the seed may not be a valid LP seed due to potential reducibility of the exchange polynomials. However, as shown by the following lemma, for an anti-symmetric quiver arising from a triangulation of $(S,M)$ we always get irreducibility.

\begin{lem}
\label{validseed}
Let $T$ be a triangulation of $(S,M)$. Then $F_j^{Q_{\overline{T}}}$ is irreducible in $\mathbb{ZP}[x_1,\ldots,x_{n}]$ for any $j$. In particular, $\Sigma_{Q_{\overline{T}}} := (\{x_1,\ldots, x_n\}, \{F_1^{Q_{\overline{T}}},\ldots, F_n^{Q_{\overline{T}}}\})$ is a valid LP seed. 

\end{lem}

\begin{proof}

The quiver $Q_{\overline{T}}$ coming from the lifted triangulation $\overline{T}$ can have at most 2 ingoing and 2 outgoing arrows at any one vertex. Hence, $gcd(b_{1j}+b_{\tilde{1}j}, \ldots, b_{mj}+b_{\tilde{m}j}) \in \{1,2,\infty\}$.

If $gcd$ is $1$ then Proposition \ref{irreducible} yields the irreducibility of $F_j^{Q_{\overline{T}}}$.

If $gcd$ is $\infty$ then $b_{ij} + b_{\tilde{i}j} = 0$ for all $i$. So $F_j^{Q_{\overline{T}}} = 2$, which is irreducible.

If $gcd$ is $2$ then due to there being at most 2 ingoing and 2 outgoing arrows at $j$ the only possibilities for $F_j^{Q_{\overline{T}}}$ are $x_i^2 +1$ and $x_i^2 + x_k^2$, which are both irreducible.
\end{proof}

Recall that the goal of this subsection has been to show triangulations fit into an LP structure by invoking Proposition \ref{quivermut}. To accomplish this we are left to prove that $\hat{F_i} = F_i$ for each $F_i$ in $\Sigma_{Q_{\overline{T}}}$. By the following lemma we may equivalently prove that the exchange polynomials in each seed $\Sigma_{Q_{\overline{T}}}$ are distinct.

\begin{lem}
\label{nonormalisation}

Let $T$ be a triangulation, $\Sigma_{Q_{\overline{T}}}$ its associated LP seed, and $i \in \{1,\ldots, n\}$. Then $\hat{F}_i = F_i$ \textit{if and only if} $F_i \neq F_j$ for any $j \neq i$.

\end{lem}

\begin{proof}

If $\hat{F_i} = F_i$ then, by definition of normalisation, for any $j \neq i$ we have $F_j$ does not divide $F_i \rvert_{x_j \leftarrow \frac{F_j}{x}}$. Hence $F_j$ does not divide $F_i$ and so, in particular, $F_i \neq F_j$.

Conversely, if $\hat{F}_i \neq F_i$ then there exists $j \neq i$ such that $F_j$ divides $F_i \rvert_{x_j \leftarrow \frac{F_j}{x}}$, which forces $x_i \notin F_j$. Suppose for a contradiction that $x_j \in F_i$. This implies the existence of a path $i \rightarrow j \rightarrow \tilde{i}$. By Proposition \ref{breakage} and Figure \ref{badvertex} we see $F_j = x_a + x_b$ and $F_i = x_j^2 + x_{a}x_{b}$. However, this contradicts $F_j$ dividing $F_i \rvert_{x_j \leftarrow \frac{F_j}{x}} = \frac{F_j^2}{x^2} + x_{a}x_{b}$. Hence $x_j \notin F_i$ and $F_j$ divides $F_i \rvert_{x_j \leftarrow \frac{F_j}{x}} = F_i$. Moreover, since $F_i$ is irreducible then $F_i = F_j$.

\end{proof}

We now list several lemmas to help discover the heterogeneity of the exchange polynomials in $\Sigma_{Q_{\overline{T}}}$.

\begin{lem}
\label{noarrows}
If $F_i = F_j$ then there are no arrows between $i$ and $j$ in $Q$.

\end{lem}

\begin{proof}

Since $x_i \notin F_i$ then $x_i \notin F_j$. As such, $b_{ij} + b_{\tilde{i}j} = 0$. Likewise, $b_{ji} + b_{\tilde{j}i} = 0$. Finally, since $b_{ij}= -b_{ji}$ and $b_{\tilde{i}j}=b_{\tilde{j}i}$, then, as required, $b_{ij} = 0$.

\end{proof}

\begin{defn}

Let $Q$ be a quiver and $V$ a set of vertices of $Q$. We say $R$ is the \textit{\textbf{$V$-restriction}} of $Q$ if $R$ consists of all arrows of $Q$ with a head or tail in $V$.

\end{defn}

\begin{lem}
\label{reverse}
Suppose $R$ is the $\{i,j\}$-restriction of $Q$ with $F_i^Q = F_j^Q$. Then the $\{i,j\}$-restriction of $\mu_{\tilde{i}} \circ \mu_{i}(R)$ is the $\{i,j\}$-restriction of $\mu_{\tilde{i}} \circ \mu_{i}(Q)$ where $F_i^{\mu_{\tilde{i}} \circ \mu_{i}(Q)} = F_j^{\mu_{\tilde{i}} \circ \mu_{i}(Q)}$. In particular, if $R$ is the $\{i,j\}$-restriction of a quiver arising from $(S,M)$ with exchange polynomials $F_i = F_j$, then so is the \{i,j\}-restriction of $\mu_{\tilde{i}} \circ \mu_{i}(R)$.
\end{lem}

\begin{proof}

By Lemma \ref{noarrows} there are no arrows between $i$ and $j$ so performing mutation at $i$ and $\tilde{i}$ in $R$ and taking the $\{i,j\}$-restriction is the same as reversing all arrows at $i$ and $\tilde{i}$ in $R$. Hence the $\{i,j\}$-restriction of $\mu_{\tilde{i}} \circ \mu_{i}(R)$ is the $\{i,j\}$-restriction of $\mu_{\tilde{i}} \circ \mu_{i}(Q)$. Moreover, the new $i^{th}$ and $j^{th}$ exchange polynomials remain unchanged, so are still equal.

\end{proof}

\begin{lem}
\label{cancellation}
Suppose $F_i = F_j$ for some $i \neq j$; $x_k \notin F_i$ for some $k$; and $i$ and $k$ are adjacent arcs in $T$. Then $(S,M)$ is either the M\"obius strip $M_4$ or the Klein bottle with one boundary component and two marked points, where neither surface has been allocated boundary variables.

\end{lem}

\begin{proof}

Under the conditions of the lemma, before cancelling 2-cycles, we must have one of the following subquivers in our adjacency quiver $Q$:

\begin{figure}[H]
\begin{center}
\includegraphics[width=10cm]{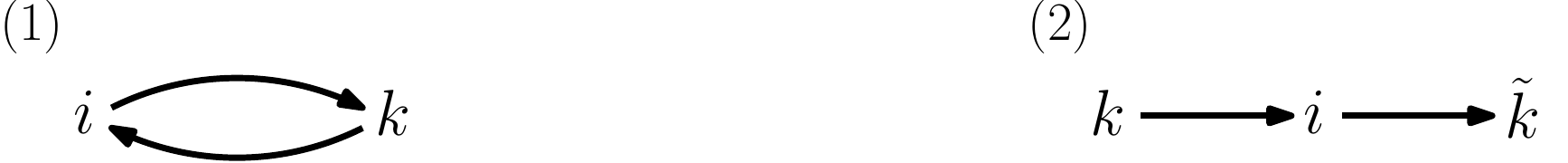}
\end{center}
\end{figure}

If the subquiver $(1)$ is in $Q$ then we must have one of the configurations shown in Figure \ref{puncturedconfiguration}. In either situation, after gluing, we obtain a punctured surface. Since we have forbidden punctures then this subquiver cannot arise from any of our triangulations. 

\begin{figure}[H]
\begin{center}
\includegraphics[width=12cm]{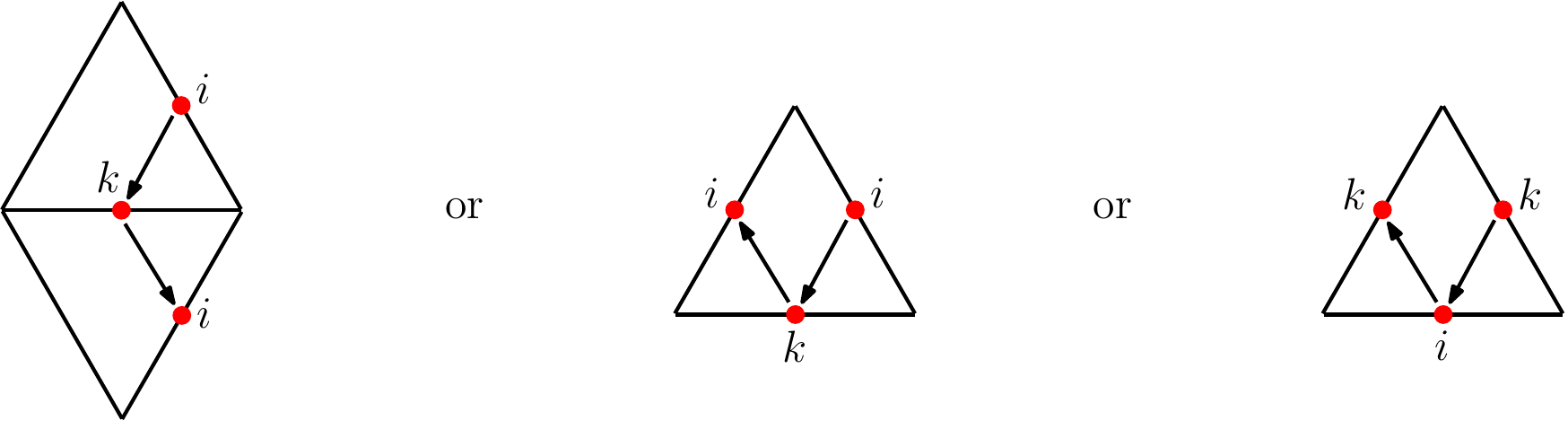}
\caption{The three configurations which produce the subquiver (1).}
\label{puncturedconfiguration}
\end{center}
\end{figure}

If the subquiver $(2)$ is in $Q$ then by Lemma \ref{breakage} we must have the following local picture shown on the left of Figure \ref{cylinderbreakage}. Note that $b$ cannot equal $a$ or $\tilde{a}$ because this would give rise to a punctured surface - the twice punctured projective space $\mathbb{R}P^2$ or the once punctured Klein bottle, respectively.

Moreover, $a$ and $b$ cannot both be boundary segments as then there is no label $j$ in the triangulation. Without loss of generality, suppose $a$ is not a boundary component. As a consequence, there is an arrow $a \rightarrow i$. Since $F_i = F_j$, using Lemma \ref{reverse}, we may assume the existence of an arrow $a \rightarrow j$. Hence we arrive at the picture shown on the right of Figure \ref{cylinderbreakage}.

\begin{figure}[H]
\begin{center}
\includegraphics[width=11cm]{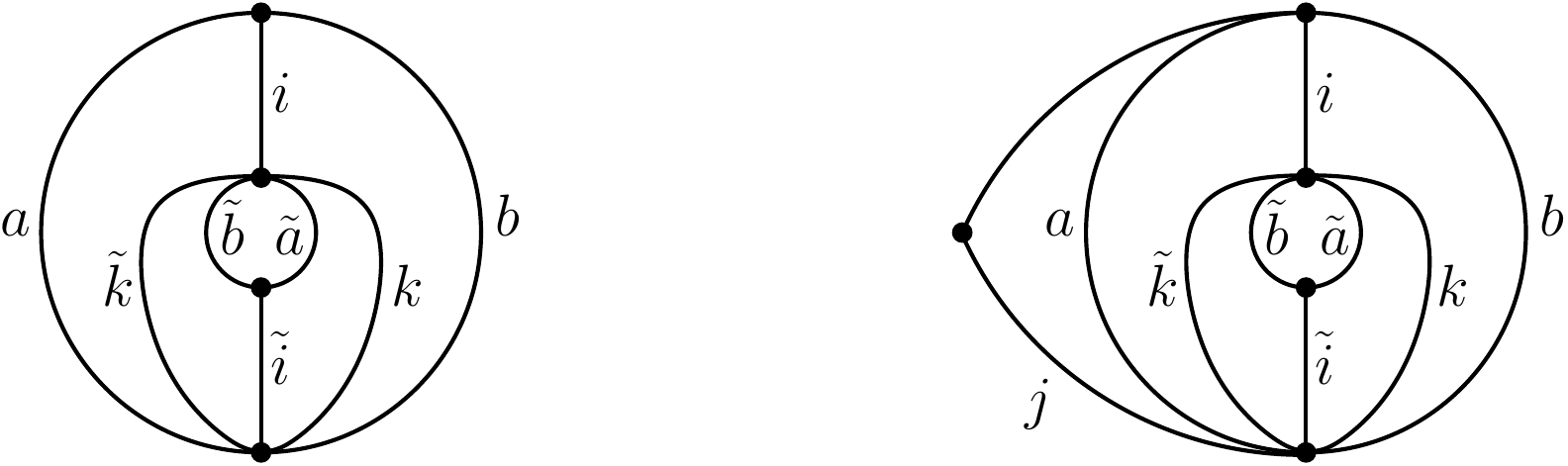}
\caption{On the right we illustrate the effect on the local configuration of the surface when there is an arrow $a \rightarrow j$ in $Q$.}
\label{cylinderbreakage}
\end{center}
\end{figure}

If $b$ is doesn't receive a variable then there is no arrow $i \rightarrow b$, and we are in one of two possible scenarios: There is a path $\tilde{m} \rightarrow j \rightarrow m$ for some $m$, or $j$ is connected to only $a$. If there is a path $\tilde{m} \rightarrow j \rightarrow m$ then by Lemma \ref{breakage} our surface must have the configuration shown on the left of Figure \ref{DoubleCancellation}. Taking the $\mathbb{Z}_2$-quotient of this yields the Klein bottle with one boundary component and $2$ marked points. Alternatively, if $j$ is connected to only $a$ then the arc $j$ is the diagonal of a square with three unlabelled boundary segments and fourth side $a$. And we obtain the surface shown on the right of Figure \ref{DoubleCancellation}. Taking the $\mathbb{Z}_2$-quotient of this yields the M\"obius strip with 4 marked points.

\begin{center}

\begin{figure}[H]
\begin{center}
\includegraphics[width=13cm]{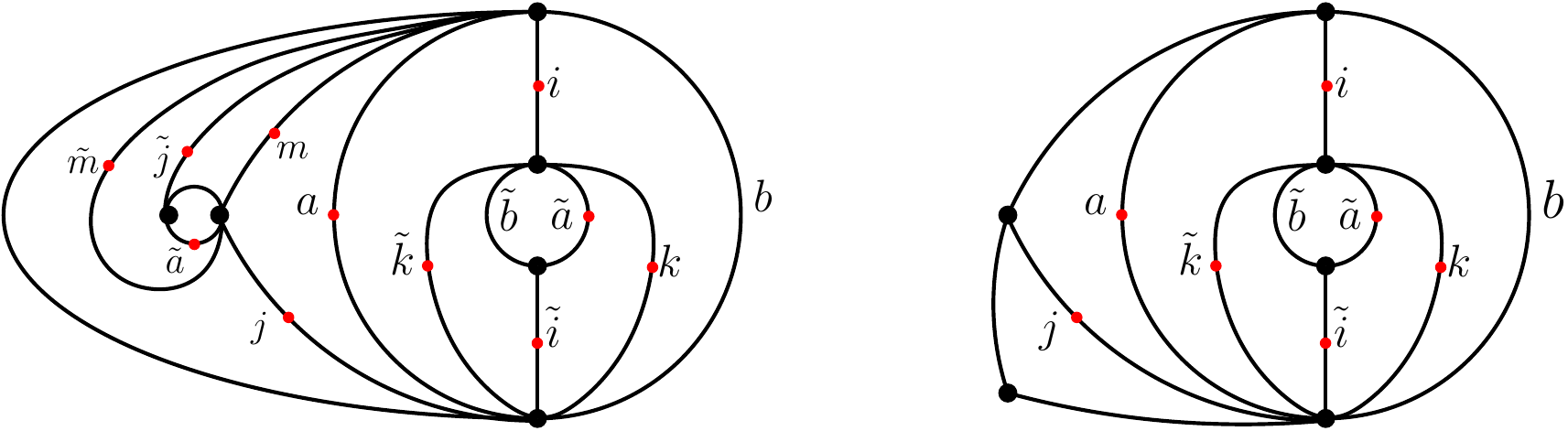}
\caption{We depict the resulting surfaces when there is either: a path $\tilde{m} \rightarrow j \rightarrow m$; or $j$ is connected only to the vertex $a$.}
\label{DoubleCancellation}
\end{center}
\end{figure}

\end{center}

If $b$ does receive a variable then there is an arrow $i \rightarrow b$ in $Q$. As such, since $F_i = F_j$, there is either an arrow $j \rightarrow b$ or an arrow $j \rightarrow \tilde{b}$. However, an arrow $j \rightarrow b$ gives rise to a punctured surface, which is forbidden. An arrow $j \rightarrow \tilde{b}$ gives rise to the configurations shown in Figure \ref{aandbwithlabels}. In both cases, taking the $\mathbb{Z}_2$-quotient again yields the Klein bottle with one boundary component and two marked points.

\begin{figure}[H]
\begin{center}
\includegraphics[width=13cm]{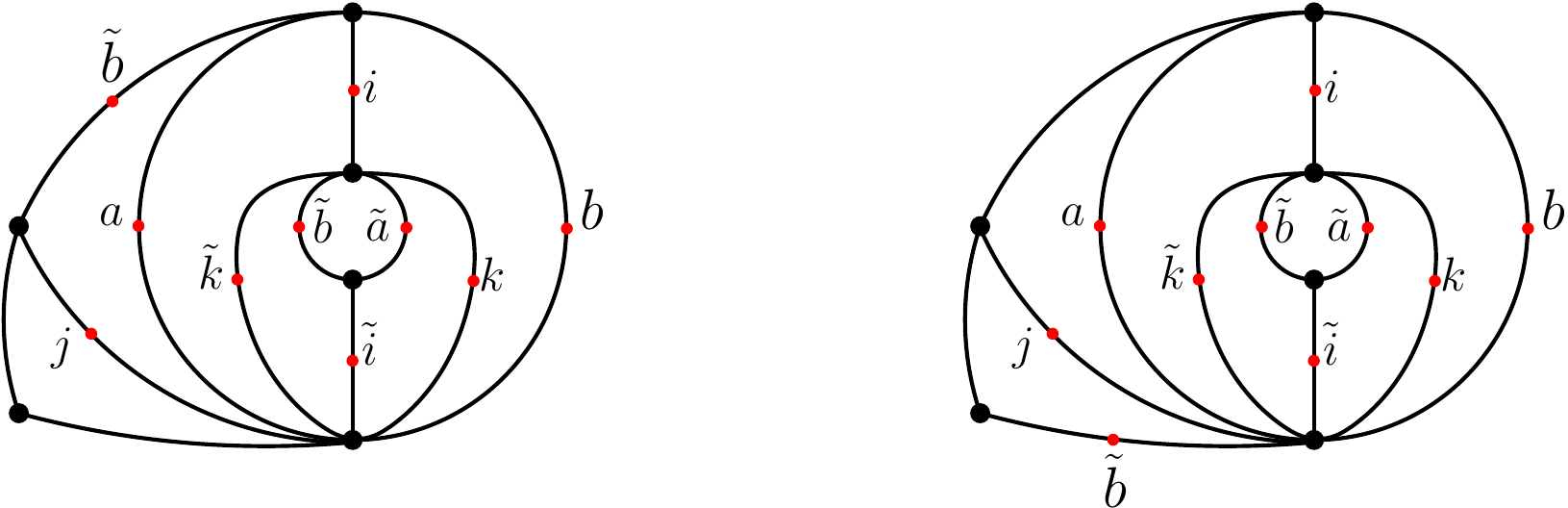}
\caption{The two possibilities of the surface when there is an arrow $j \rightarrow \tilde{b}$.}
\label{aandbwithlabels}
\end{center}
\end{figure}

\end{proof}

\begin{lem}
\label{weight}
If $F_i = F_j$ then the quiver $Q$ cannot contain either of the subquivers $ \tilde{k} \leftarrow i \rightarrow k $ or $i \stackrel{2}{\longrightarrow} k$, for any vertex $k$ of $Q$.

\end{lem}

\begin{proof}

If $ \tilde{k} \leftarrow i \rightarrow k $ is a subquiver of $Q$ then antisymmetry implies the existence of the path $i \rightarrow k \rightarrow \tilde{i}$. Therefore, by Lemma \ref{breakage}, we have the sub triangulation shown in Figure \ref{doublearrows}. Since $F_i = F_j$ then there must be an arrow $j \rightarrow k$ or $j \rightarrow \tilde{k}$. However, any triangle with side $k$ or $\tilde{k}$ also has a side $i$ or $\tilde{i}$. This forces an arrow between $i$ and $j$ or $i$ and $\tilde{j}$, contradicting Lemma \ref{noarrows}.
If $i \stackrel{2}{\longrightarrow} k$ is a subquiver of $Q$ then since $F_i = F_j$, without loss of generality, $i \stackrel{2}{\longrightarrow} k \stackrel{2}{\longleftarrow} j$ is a subquiver of $Q$. However, this contradicts the fact that any vertex in $Q$ can have at most 2 incoming arrows.

\end{proof}

\begin{figure}[H]
\begin{center}
\includegraphics[width=3.5cm]{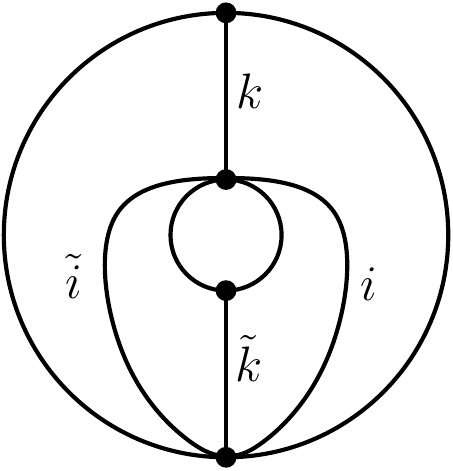}
\caption{The local configuration of the surface when there is a path $\tilde{k} \leftarrow i \rightarrow k$.}
\label{doublearrows}
\end{center}
\end{figure}

\begin{lem}
\label{equals}
Let $T$ be a triangulation and $\Sigma_{Q_{\overline{T}}}$ its associated LP seed. Then $\hat{F}_i = F_i$ for any $i \in \{1,\ldots, n\}$.

\end{lem}

\begin{proof}

By Lemma \ref{nonormalisation} it suffices to show that $F_j \neq F_i$ for any $j \neq i$. Now, by Lemma \ref{noarrows} we know there are no arrows between $i$ and $j$. Due to Lemma \ref{weight} we know there are no arrows of weight greater than $1$ in the $\{i,j\}$-restriction of $Q$. Furthermore, by Lemma \ref{cancellation} and Lemma \ref{weight} if $i$ (or $j$) is connected to both $k$ and $\tilde{k}$ for some vertex $k$ in $Q$, then the corresponding surface must be either the M\"obius strip $M_4$ or the Klein bottle with one boundary component and two marked points, where neither surface has been allocated boundary variables. Having dealt with these cases, from here on we may therefore assume $i$ and $j$ are connected to at most one of $k$ and $\tilde{k}$ for any vertex $k$ in $Q$. After reversing all arrows at $i$ if needed, $i$ and $j$ will locally have the same quiver up to exchanging $a$ and $\tilde{a}$. I.e. If $i \leftarrow k$ (or $i \rightarrow k$) then $j \leftarrow k$ ($j \rightarrow k$) or $j \leftarrow \tilde{k}$ ($j \rightarrow \tilde{k}$).  \newline

To determine the remaining surfaces which emit triangulations with $F_i = F_j$ we will split our task into four cases depending on whether $i$ and $j$ are connected to precisely $1$, $2$, $3$ or $4$ vertices. After exchanging the roles of $j$ and $\tilde{j}$ if necessary, we may assume there are arrows $i \leftarrow a$ and $j \leftarrow a$ for some fixed vertex $a$. Furthermore, note that in the quivers we draw we only include arrows between $i$ and $j$. For each of these quivers $R$ we are asking which triangulations $T$ of $(S,M)$ have the property that the $\{i,j\}$-restriction of $Q_{\overline{T}}$ is $R$. \newline

\noindent \underline{\textbf{Case 1}}: $i$ and $j$ are connected to precisely one vertex. \newline

The only such quiver for this case is $i \leftarrow a \rightarrow j$. Since $i$ and $j$ are not connected to any other vertex, the arcs $i$ and $j$ are the diagonals of quadrilaterals with three boundary segments and fourth side $a$. This yields the $6$-gon shown in Figure \ref{6gon}.

\begin{figure}[H]
\begin{center}
\includegraphics[width=4cm]{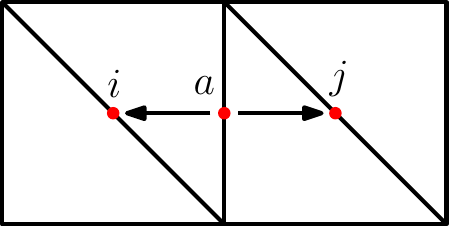}
\caption{The $6$-gon without boundary variables - the only surface emitting a triangulation whose adjacency quiver both contains the subquiver $i \leftarrow a \rightarrow k$; and has $F_i = F_j$.}
\label{6gon}
\end{center}
\end{figure}

\noindent \underline{\textbf{Case 2}}: $i$ and $j$ are connected to precisely two vertices. \newline

The possible subquivers for this case are listed in Figure \ref{case2quivers}.

\begin{figure}[H]
\begin{center}
\includegraphics[width=12cm]{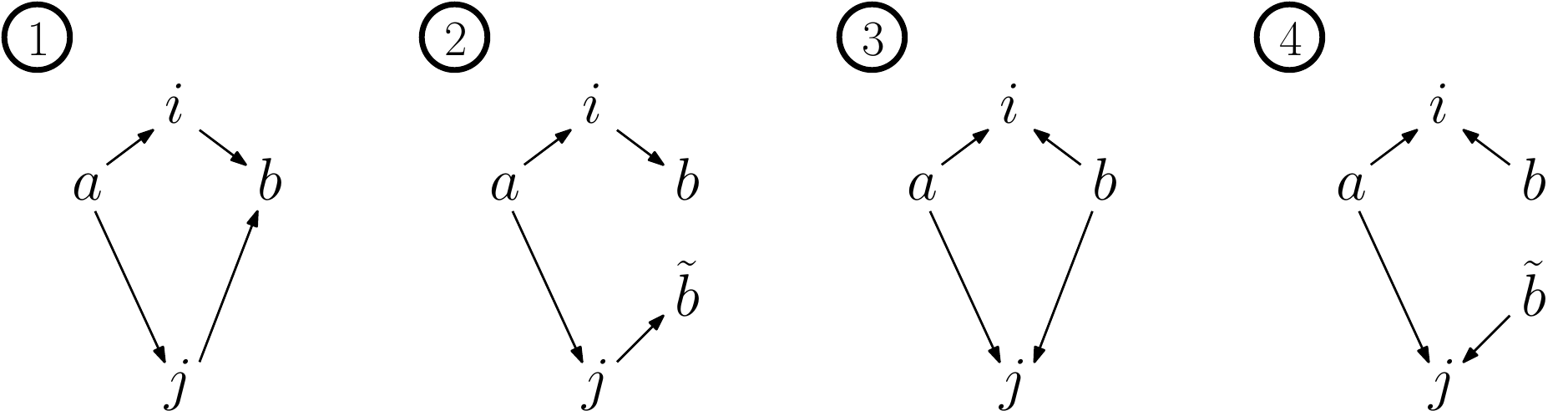}
\caption{The list of the possible $\{i,j\}$-restriction quivers when $i$ and $j$ are connected to precisely two vertices, and $F_i = F_j$.}
\label{case2quivers}
\end{center}
\end{figure}

For each of the subquivers listed in Figure \ref{case2quivers} we present below the possible triangulations/surfaces that produce them. To elaborate, we use the quiver to determine the conceivable adjacencies of triangles in the triangulation, and this is how the surface is reconstructed.

\begin{figure}[H]
\begin{center}
\includegraphics[width=13cm]{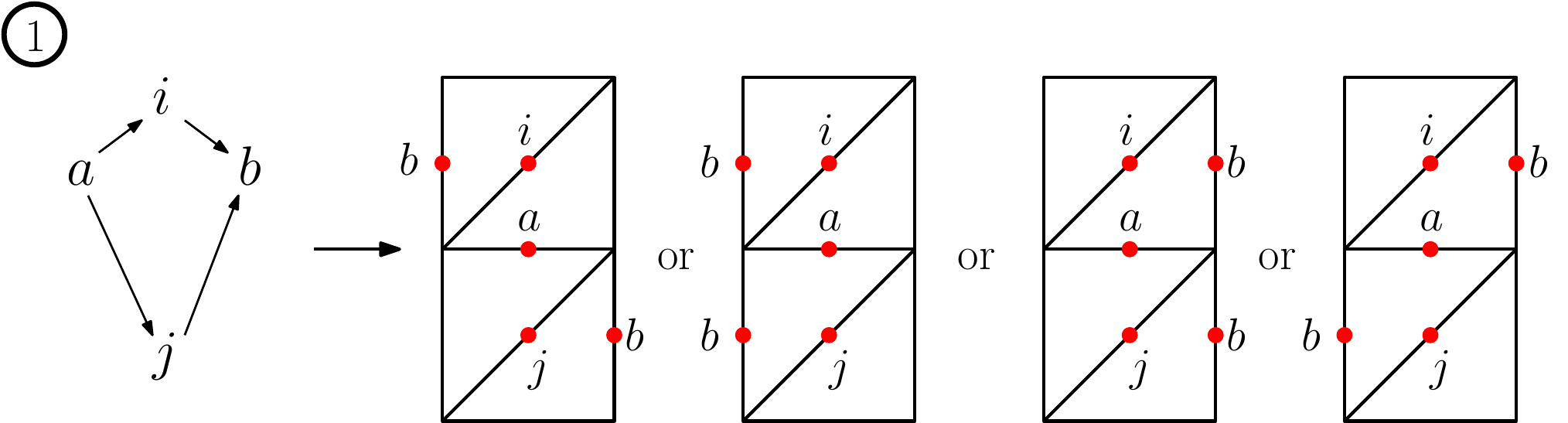}

\end{center}
\end{figure}

\begin{figure}[H]
\begin{center}
\includegraphics[width=13cm]{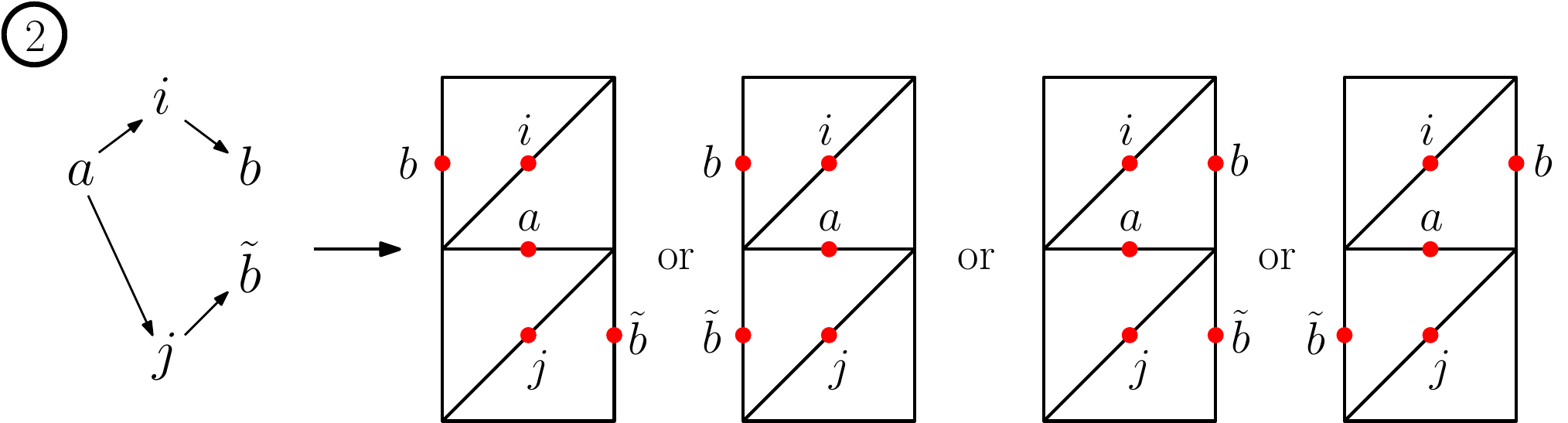}

\end{center}
\end{figure}

\begin{figure}[H]
\begin{center}
\includegraphics[width=13cm]{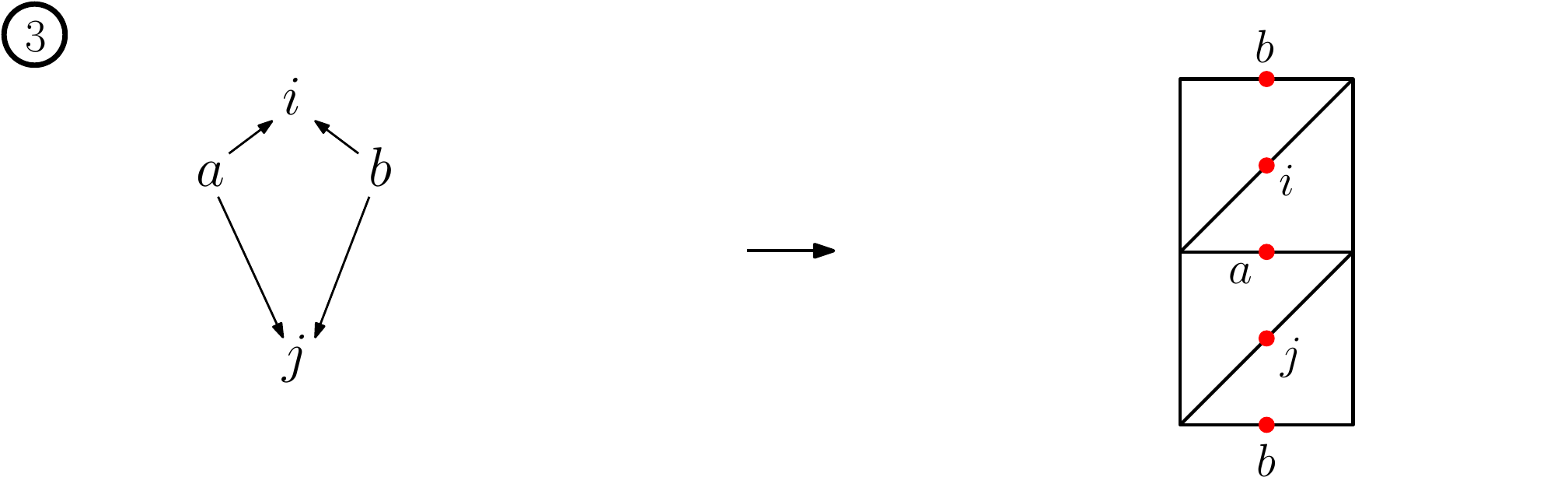}

\end{center}
\end{figure}

\begin{figure}[H]
\begin{center}
\includegraphics[width=13cm]{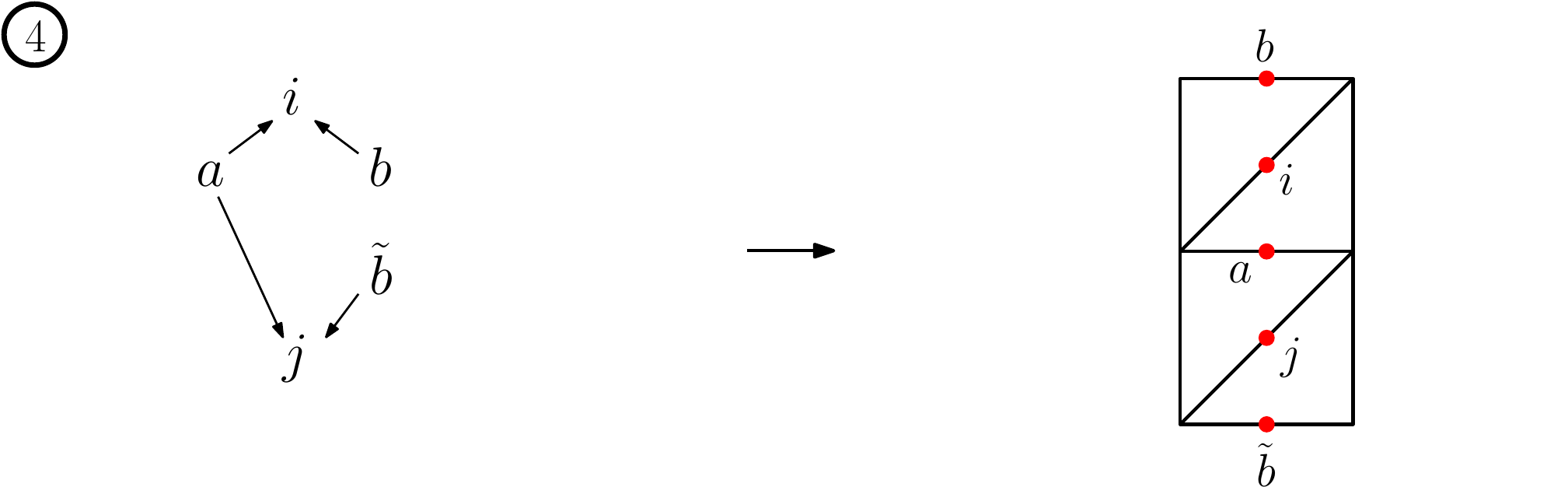}
\caption{Upon gluing and taking $\mathbb{Z}_2$-quotients, the (unpunctured) surfaces we obtain in Case $2$ are: The cylinder with $2$ marked points on each boundary component, and the M\"obius strip $M_4$.}
\label{case2surfaces}
\end{center}
\end{figure}

For each of these Case $2$ quivers we list the surfaces obtained after glueing and taking the $\mathbb{Z}_2$-quotient. 

\begin{enumerate}[label=\protect\circled{\arabic*}]
\item The first and fourth give the cylinder with two marked points on each boundary component; the second and third give the once punctured square.
\item All produce the M\"obius strip with four marked points.
\item The cylinder with two marked points on each boundary component.
\item The M\"obius strip with four marked points.
\end{enumerate}

\noindent \underline{\textbf{Case 3}}: $i$ and $j$ are connected to precisely three vertices. \newline

The possible subquivers for this case are listed in Figure \ref{case3quivers}. Here we are using the fact that there cannot be more than two incoming/outgoing arrows at any given vertex.

\begin{figure}[H]
\begin{center}
\includegraphics[width=12cm]{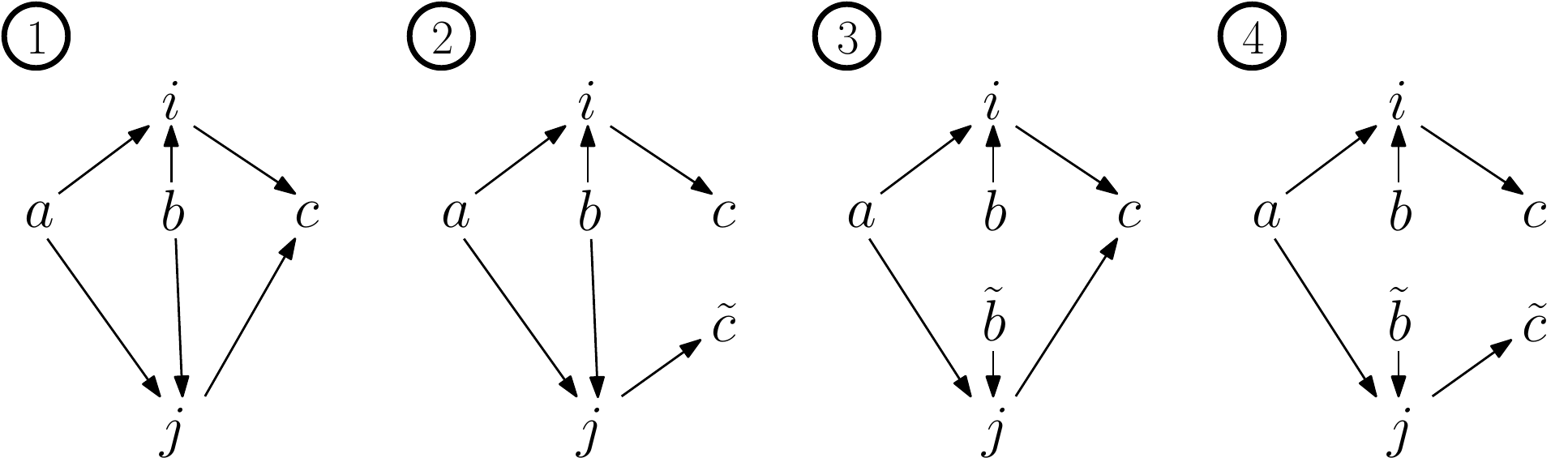}
\caption{The list of the possible $\{i,j\}$-restriction quivers when $i$ and $j$ are connected to precisely three vertices, and $F_i = F_j$.}
\label{case3quivers}
\end{center}
\end{figure}

Note that it suffices to check only subquivers $1$, $2$ and $3$ since $4$ is equivalent to $3$ after swapping the roles of $a$ and $b$ and using anti-symmetry. Below we present the possible surfaces producing the subquivers $1$, $2$ and $3$.

\begin{figure}[H]
\begin{center}
\includegraphics[width=12cm]{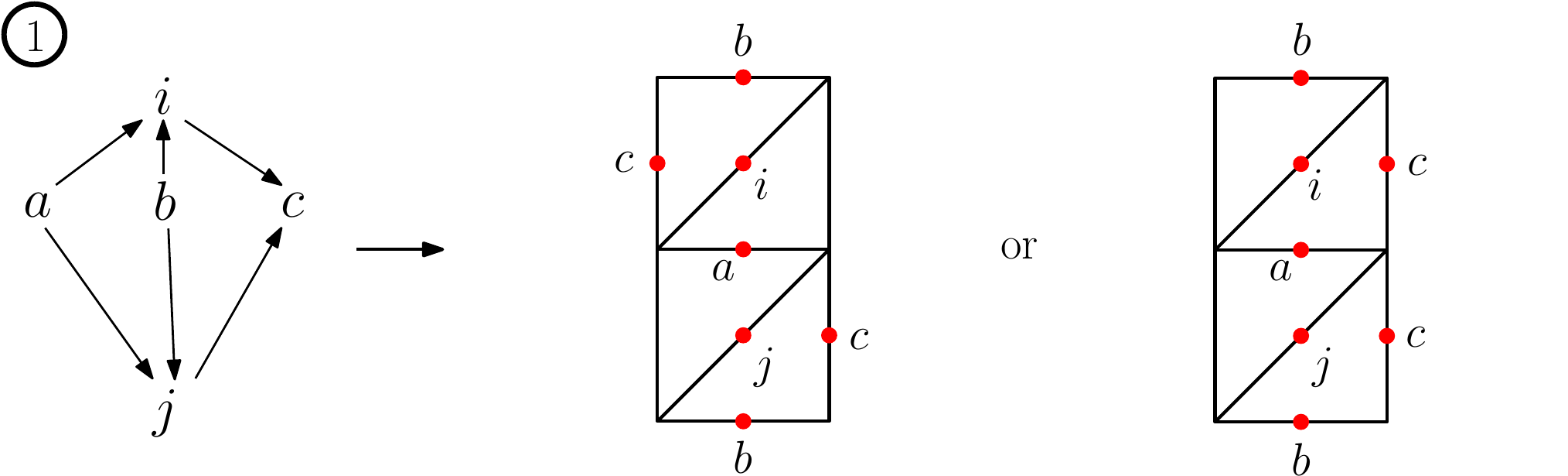}

\end{center}
\end{figure}

\begin{figure}[H]
\begin{center}
\includegraphics[width=12cm]{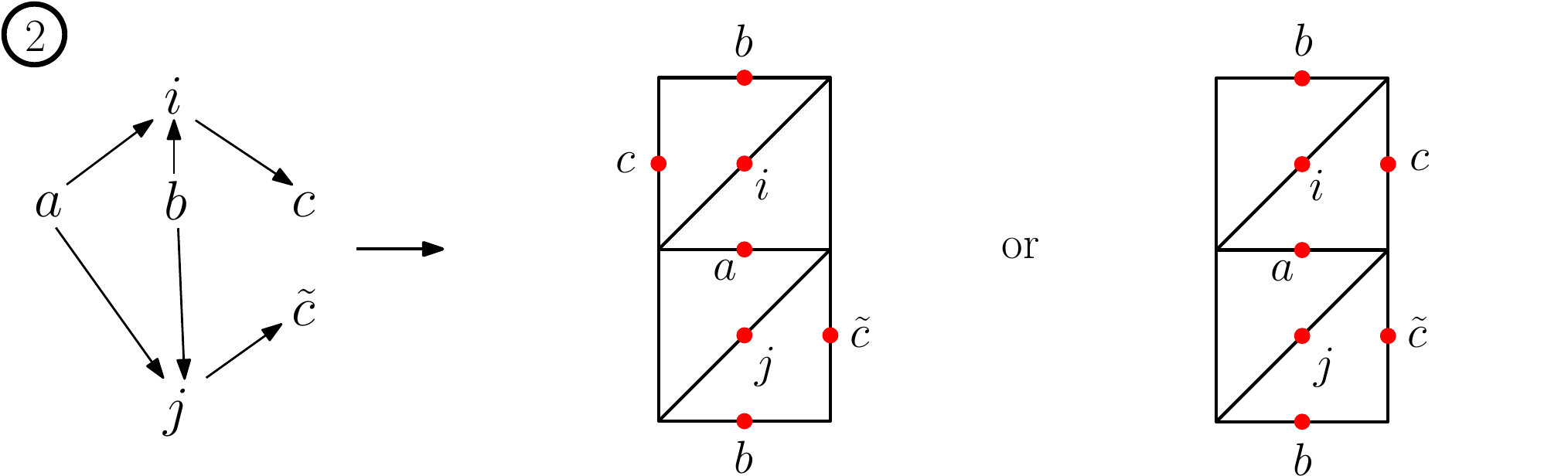}

\end{center}
\end{figure}

\begin{figure}[H]
\begin{center}
\includegraphics[width=12cm]{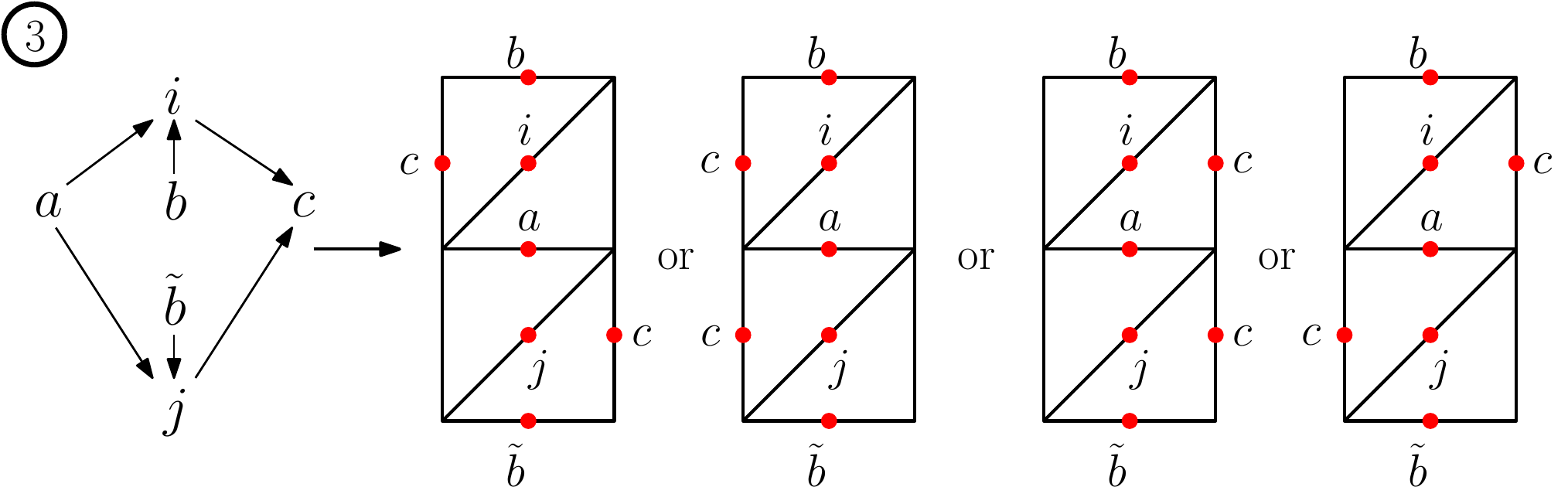}
\caption{Upon gluing and taking $\mathbb{Z}_2$-quotients, the (unpunctured) surfaces we obtain in Case $3$ are: The torus and Klein bottle, both with $1$ boundary component and $2$ marked points.}
\label{case3surfaces}
\end{center}
\end{figure}

For each of these Case $3$ quivers we list the surfaces obtained after glueing and taking the $\mathbb{Z}_2$-quotient.

\begin{enumerate}[label=\protect\circled{\arabic*}]
\item The first gives the torus with one boundary component and two marked points; the second produces the once punctured digon.
\item The first gives the the Klein bottle with one boundary component and two marked points; the second produces the once punctured M\"obius strip with two marked points.
\item The first and fourth give the Klein bottle with one boundary component and two marked points; the second and third produce the once punctured M\"obius strip with two marked points.
\end{enumerate}

\noindent \underline{\textbf{Case 4}}: $i$ and $j$ are connected to precisely four vertices. \newline

Being connected to four vertices the arc $i$ will be the diagonal of a square with sides $a$, $b$, $c$ and $d$. The arc $j$ will therefore be the diagonal of a square with sides possessing labels from the set $\{a,\tilde{a},b,\tilde{b},c,\tilde{c},d,\tilde{d}\}$. After gluing and taking the $\mathbb{Z}_2$-quotient then, if this procedure creates a surface, it will be a closed surface. However, we have forbidden punctured surfaces so none of our permitted surfaces satisfy Case $4$. \newline

In summary, the only unpunctured surfaces emitting triangulations producing non-distinct exchange polynomials are: the $6$-gon; the M\"obius strip with four marked points; the cylinder with two marked points on each boundary component; and the torus and the Klein bottle, both with one boundary component and two marked points. It is important to note that these surfaces only produce non-distinct exchange polynomials when their boundary segments receive no variables. In this paper we only consider unpunctured surfaces receiving boundary variables, therefore, any triangulation of our surfaces will yield a distinct collection of exchange polynomials.

\end{proof}

\begin{prop}
\label{partial}
Let $i$ be a t-mutable arc in a triangulation $T$ of $(S,M)$. Then flipping $i$ in $T$ corresponds to $LP$ mutation at $i$ of the associated seed $\Sigma_{Q_{\overline{T}}} := (\{x_1,\ldots, x_n\}, \{F_1^{Q_{\overline{T}}},\ldots, F_n^{Q_{\overline{T}}}\})$.

\end{prop}

\begin{proof}

By Lemmas \ref{correctpolys} and \ref{equals} we obtain that LP and quasi-cluster mutation agree on the level of variable change. Moreover, Lemma \ref{breakage} tells us that if $i$ is a t-mutable arc in $T$ then there is no path $a \rightarrow i \rightarrow \tilde{a}$ in $Q_{\overline{T}}$ for any vertex $a$. Lemmas \ref{validseed} and \ref{equals} confirm that $Q_{\overline{T}}$ is a valid seed and $F_1^{Q_{\overline{T}}} = \hat{F}_1^{Q_{\overline{T}}}$ for each exchange polynomial of $\Sigma_{Q_{\overline{T}}}$. Therefore we may evoke Proposition \ref{quiverandflip} to verify that double mutation at $i$ and $\tilde{i}$ in $Q_{\overline{T}}$ coincides with LP mutation at $i$, for each t-mutable arc $i$ in $T$. Finally, since Proposition \ref{quiverandflip} tells us that double mutation at $i$ and $\tilde{i}$ corresponds to flipping the arc $i$ in $T$, then the proof is complete. 

\end{proof}

\noindent \textbf{Remark}: Note that the LP seed $\{(a,1+b),(b,a+c),(c,1+b)\}$ in Example 4.7, \cite{lam2012laurent} fails to agree with cluster algebra mutation because it arises from the $6$-gon without any boundary variables. We present more of a discussion about this in Subsection \ref{punctures}.

\begin{subsection}{Proof of the main theorem.}

\begin{thm}
\label{main}
Let $(S,M)$ be an unpunctured (orientable or non-orientable) marked surface. Then the LP cluster complex $\Delta_{LP}(S,M)$ is isomorphic to the quasi-arc complex $\Delta^{\otimes}(S,M)$, and the exchange graph of $\mathcal{A}_{LP}(S,M)$ is isomorphic to $E^{\otimes}(S,M)$. 

More explicitly, let $T$ be a quasi-triangulation of $(S,M)$ and $\Sigma_{T}$ its associated LP seed. Then in the LP algebra $\mathcal{A}_{LP}(\Sigma_{T})$ generated by this seed the following correspondence holds:
\begin{align*}
&\hspace{8mm} \mathbf{\mathcal{A}_{LP}(\Sigma_T)} & & &\mathbf{(S,M)} \hspace{20mm}&  \\ 
&\textit{Cluster variables} &\longleftrightarrow& &\textit{Lambda lengths of quasi-arcs} & \\
&\hspace{8mm}\textit{Clusters}  &\longleftrightarrow& &\textit{Quasi-triangulations} \hspace{7mm}& \\
&\hspace{4mm} \textit{LP mutation}   &\longleftrightarrow&  &\textit{Flips} \hspace{21.5mm}& \\
\end{align*}

\end{thm}

\begin{proof}

By Proposition \ref{partial} all that is left to show is that LP mutation coincides with quasi-cluster mutation when: 

\begin{enumerate}[label=(\alph*)]
\item we flip an arc in a triangulation to a one-sided closed curve.
\item we flip quasi-arcs in quasi-triangulations containing a one-sided closed curve.
\end{enumerate}

\noindent \underline{\textbf{Case (a)}}. \newline

To resolve case (a) it suffices to show that flipping the arc $a$ in Figure \ref{casea} agrees with LP mutation at $a$ of the associated seed. 

LP mutation at $a$ produces the exchange polynomials: \newline

$F_a' = F_a$ \hspace{2mm} $F_b' = (c+d)^2 + a'^2cd$ \hspace{2mm}$F_c' = dy + a'bw$ \hspace{2mm} $F_d' = cz + a'bx$. \newline

A simple computation produces the associated normalised exchange polynomials, which are recorded below. These normalised polynomials do indeed describe how lengths of arcs in the flipped quasi-triangulation exchange, so case (a) has been verified. \newline

$\hat{F}_a' = F_a$ \hspace{2mm} $\hat{F}_b' = \frac{(c+d)^2 + a'^2cd}{a'^2}$ \hspace{2mm}$\hat{F}_c' = dy + a'bw$ \hspace{2mm} $\hat{F}_d' = cz + a'bx$  \newline

\begin{figure}[H]
\begin{center}
\includegraphics[width=9cm]{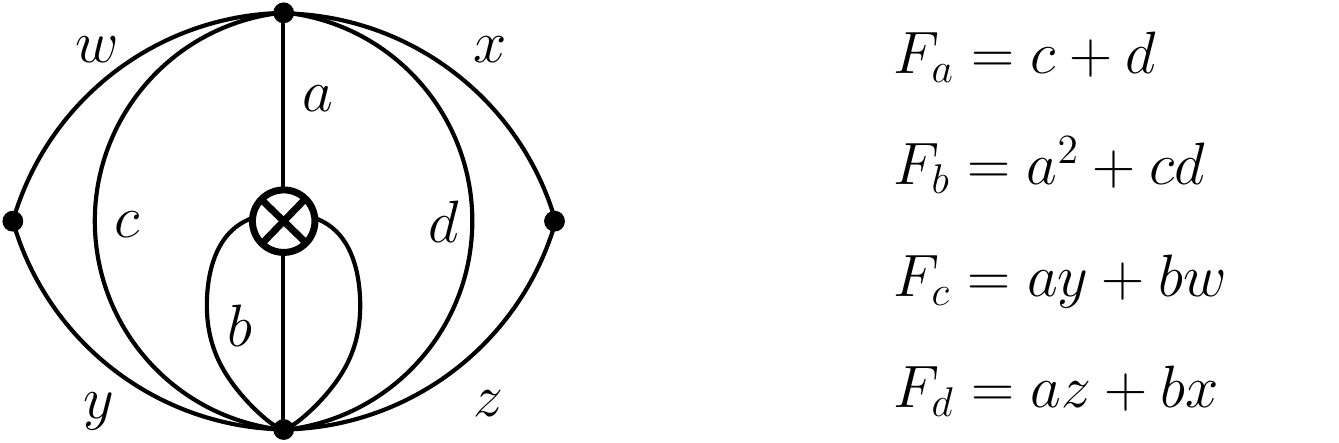}
\caption{A triangulation together with the associated exchange polynomials.}
\label{casea}
\end{center}
\end{figure}

\noindent \underline{\textbf{Case (b)}}. \newline

\noindent We split the task of verifying case (b) into four subcases:

\begin{enumerate}

\item Flipping a quasi-arc that is not enclosed in a region containing a one-sided closed curve.

\item Flipping $b, c$ or $d$ in the triangulation on the left of Figure \ref{subcaseb}.

\item Flipping $a$ in the middle triangulation of Figure \ref{subcaseb}.

\item Flipping $b, d$ or $y$ in the triangulation on the right of Figure \ref{subcaseb}.

\end{enumerate}

\begin{figure}[H]
\begin{center}
\includegraphics[width=13cm]{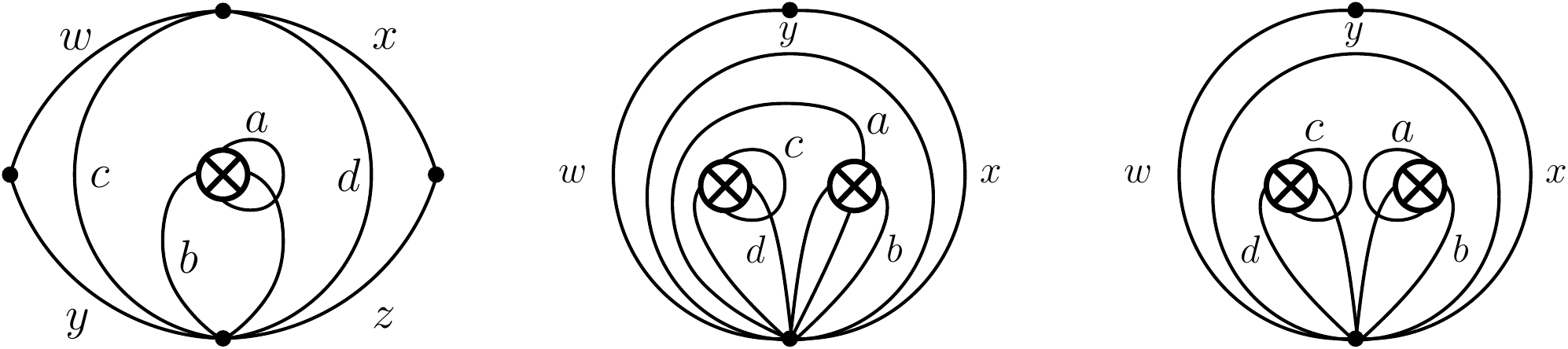}
\caption{The three types of (local) configurations that contain a one-sided closed curve.}
\label{subcaseb}
\end{center}
\end{figure}

\noindent \underline{Subcase $1$}: Here LP mutation and surface flips coincide due to Proposition \ref{partial} and Case (a). \newline

\noindent \underline{Subcase $2$}: To verify that LP mutation and surface flips coincide for this case, it suffices to check mutation at $b$ and $c$. \newline

The exchange polynomials corresponding to the left triangulation in Figure \ref{subcaseb} are:

\begin{center}
$F_a = c + d$ \hspace{2mm} $F_b = (c+d)^2 + a^2cd$ \hspace{2mm}$F_c = dy + abw$ \hspace{2mm} $F_d = cz + abx$.
\end{center}

\noindent Mutating at $b$ produces the following exchange polynomials:

\begin{center}
$F_a' = F_a$ \hspace{2mm} $F_b' = F_b$ \hspace{2mm}$F_c' = ab'y + dw$ \hspace{2mm} $F_d' = ab'z + cx$.
\end{center}

\noindent If instead we mutate at $c$ we obtain the following exchange polynomials:

\begin{center}
$F_a' = y+c'$ \hspace{2mm} $F_b' = (y+c')^2 + a^2yc'$ \hspace{2mm}$F_c' = F_c$ \hspace{2mm} $F_d' = wz + xc'$.
\end{center}

The normalised versions of both of these sets of polynomials describe how lengths of arcs transform in their respective quasi-triangulations, so this completes subcase $2$. \newline

Subcases $3$ and $4$ hold analogous to case $(a)$ and subcase $2(b)$, respectively.

\end{proof}

\end{subsection}

\begin{subsection}{Punctured surfaces.}
\label{punctures}

We confess now that we have omitted punctured surfaces throughout this paper on account of their failure to emit an LP structure that encompasses the cluster structure already established (on orientable surfaces) in \cite{fomin2008cluster}. The reason why the flip/length structure of a punctured surface cannot be imitated by an LP structure is simple; if a surface is punctured then it emits a tagged triangulation containing two (distinct) arcs whose plain versions coincide. These two arcs have identical exchange polynomials, so by Lemma \ref{nonormalisation} the normalised exchange polynomials differ from the exchange polynomials. This ensures the LP structure and the quasi-cluster structure will not coincide.

Recall that when the boundary segments receive no variables the $6$-gon and the cylinder $C_{2,2}$ have the same cluster structure as the punctured triangle and the twice punctured monogon, respectively - see Figure \ref{puncturedsurfaces}. From the comments made above we instantly get confirmation of the fact obtained in the proof of Lemma \ref{equals}, that in the absence of boundary variables, there is no LP algebra producing the cluster structure of the $6$-gon or the cylinder $C_{2,2}$. One might be tempted to believe the torus with one boundary component and two marked points follows suit, and shares its cluster structure with a punctured surface, however, the work of Bucher, Yakimov \cite{bucher2016recovering} and Gu \cite{gu2011graphs} tells us that this is not the case.

\begin{figure}[H]
\begin{center}
\includegraphics[width=12cm]{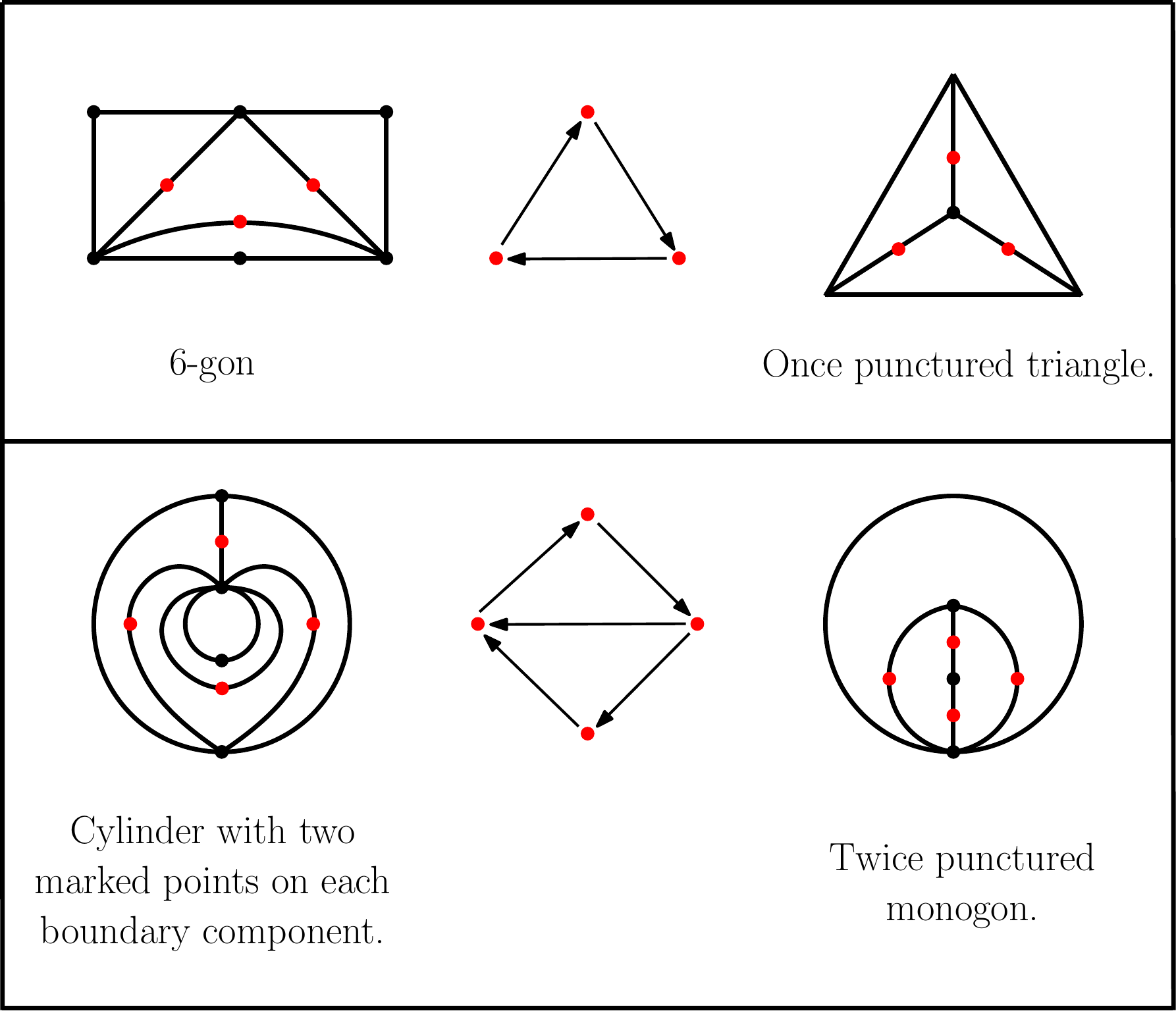}
\caption{Here we list all orientable bordered surfaces which share their cluster algebra structure with a punctured surface. For each of these bordered surfaces we provide the punctured surface possessing the same cluster structure. In each case we present triangulations emitting matching adjacency quivers.}
\label{puncturedsurfaces}
\end{center}
\end{figure}

\end{subsection}

\nocite{*}

\bibliography{laurent}
\bibliographystyle{plain}

\Addresses

\end{document}